\documentclass[11pt, english, reqno]{amsart}

\usepackage{wasysym, amsmath, amsfonts, amssymb, amscd, enumerate, young, youngtab, empheq, stmaryrd, slashed}
\usepackage{babel}
\usepackage{tikz}
\usepackage{hyperref}
\usepackage{here}
\usepackage{amssymb}
\usepackage{stmaryrd}
\usepackage{pdflscape}
\usepackage[graphicx]{realboxes}
\usepackage[figuresright]{rotating}
\newcommand{\pie}[1]{%
\begin{tikzpicture}
\tiny\draw circle  (0.9ex);\fill[rotate=270, gray!90!white] (0.9ex,0) arc (0:#1:0.9ex) -- (0,0) -- cycle;
\end{tikzpicture}%
}
\pgfdeclarepatternformonly{mygrid}{\pgfqpoint{-1pt}{-1pt}}{\pgfqpoint{40pt}{40pt}}{\pgfqpoint{2.0pt}{2.0pt}}%
{
  \pgfsetlinewidth{0.4pt}
  \pgfpathmoveto{\pgfqpoint{0pt}{0pt}}
  \pgfpathlineto{\pgfqpoint{0pt}{3.1pt}}
  \pgfpathmoveto{\pgfqpoint{0pt}{0pt}}
  \pgfpathlineto{\pgfqpoint{3.1pt}{0pt}}
  \pgfusepath{stroke}
}
\usetikzlibrary{arrows,decorations.markings,positioning, shapes, patterns, plotmarks, shadows}
\tikzset{inner sep=0pt, node distance=5mm,
  root/.style={circle,draw,minimum size=5pt},
  broot/.style={circle,draw,minimum size=5pt, fill},
  xroot/.style={circle,draw,minimum size=5pt,fill=gray!70!white},
		crossroot/.style={circle,draw,minimum size=7pt,label=below:$\times$},
  doublearrow/.style={postaction={decorate},   decoration={markings,mark=at position .6 with {\arrow[line width=1.2pt]{>}}},double distance=1.6pt,thick},
  doublenoarrow/.style={double distance=1.6pt,thick},
  rdoublearrow/.style={postaction={decorate},   decoration={markings,mark=at position .4 with {\arrowreversed[line width=1.2pt]{>}}},double distance=1.6pt,thick},
  rtriplearrow/.style={postaction={decorate},   decoration={markings,mark=at position .4 with {\arrowreversed[line width=1.2pt]{>}}},double distance=2.5pt,thick},
	curvedline/.style={bend=right}
	} 
\theoremstyle{plain}
\newtheorem{theo}{Theorem}[section]
\newtheorem{lem}[theo]{Lemma}

\theoremstyle{definition}

\newtheorem{definition}[theo]{Definition}

\theoremstyle{plain}
\newtheorem{lemma}[theo]{Lemma}
\newtheorem{theorem}[theo]{Theorem}
\newtheorem{corollary}[theo]{Corollary}
\newtheorem{proposition}[theo]{Proposition}

\theoremstyle{definition}

\newtheorem{remark}[theo]{Remark}

\newcommand{\beq}{\begin{equation}}
\newcommand{\eeq}{\end{equation}}
\renewcommand{\a}{\alpha}

\renewcommand{\d}{\delta}
\newcommand{\e}{\epsilon}

\renewcommand{\l}{\lambda}
\newcommand{\m}{\mu}
\def\LSA{Lie  superalgebra}

\renewcommand{\L}{\Lambda}

\newcommand{\out}{\operatorname{out}}

\newcommand{\bC}{\mathbb{C}}

\newcommand{\bR}{\mathbb{R}}
\newcommand{\bZ}{\mathbb{Z}}

\renewcommand{\max}{\rm{max}}

\renewcommand{\gg}{\mathfrak{g}}
\newcommand{\gh}{\mathfrak{h}}

\newcommand{\gm}{\mathfrak{m}}

\newcommand{\gs}{\mathfrak{s}}
\newcommand{\gt}{\mathfrak{t}}

\newcommand{\ggl}{\mathfrak{gl}}
\newcommand{\oder}{\operatorname{der}}
\declareslashed{}{/}{0}{0}{\partial}

\newcommand\GL{\mathrm{GL}}

\renewcommand\sl{\mathfrak{sl}}

\newcommand{\cF}{\mathcal{F}}

\renewcommand{\square}{\kern1pt\vbox
{\hrule height 0.6pt\hbox{\vrule width 0.6pt\hskip 3pt
\vbox{\vskip 6pt}\hskip 3pt\vrule width 0.6pt}\hrule height0.6pt}\kern1pt}

\DeclareMathOperator\Aut{Aut}

\DeclareMathOperator\ad{ad}
\DeclareMathOperator\vol{vol}
\DeclareMathOperator\Id{Id}

\newcommand\Ker{\operatorname{Ker}}

\newcommand{\wt}{\widetilde}

\newcommand{\soc}{\operatorname{soc}}
\newcommand{\n}{\nabla}

\newcommand{\be}{\begin{equation}}
\newcommand{\ee}{\end{equation}}

\def\<#1,#2>{\langle\,#1,\,#2\,\rangle}
\newcommand{\arr}{\begin{array}{rlll}}
\newcommand{\ea}{\end{array}}
\newcommand{\bea}{\begin{eqnarray}}
\newcommand{\eea}{\end{eqnarray}}
\newcommand{\bean}{\begin{eqnarray*}}
\newcommand{\eean}{\end{eqnarray*}}
\def\sideremark#1{\ifvmode\leavevmode\fi\vadjust{
\vbox to0pt{\hbox to 0pt{\hskip\hsize\hskip1em
\vbox{\hsize3cm\tiny\raggedright\pretolerance10000
\noindent #1\hfill}\hss}\vbox to8pt{\vfil}\vss}}}
\newcounter{ssig}
\newcounter{ttig}
\title[Homogeneous supermanifolds and Lie superalgebras]
{Homogeneous irreducible supermanifolds\\ and graded Lie superalgebras}
\author{D.\ V.\ Alekseevsky, A.\ Santi}
\address{Dmitri V.\ Alekseevsky, 
A.\ A.\ Kharkevich Institute for Information Transmission Problems,
 B. Karetnui  per. 19, 127051, Moscow, Russia and Faculty of Science,
 University of Hradec Kr\'alov\'e, Rokitansk\'eho 62, Hradec Kr\'alov\'e, 50003, Czech Republic}
\email{dalekseevsky@iitp.ru}
\address{Andrea Santi,
Maxwell Institute and School of Mathematics, The University
  of Edinburgh, James Clerk Maxwell Building, Peter Guthrie Tait Road,
  Edinburgh EH9 3FD, Scotland, UK}
\email{asanti.math@gmail.com}
\thanks{EMPG-15-20}
\newcommand{\0}{\overline{0}}
\newcommand{\1}{\overline{1}}
\begin{document}
\begin{abstract} 
A  depth one grading  $\mathfrak{g}= \mathfrak{g}^{-1}\oplus \mathfrak{g}^0 \oplus \mathfrak{g}^1 \oplus \cdots \oplus \mathfrak{g}^{\ell}$   of  a  finite  dimensional  Lie  superalgebra  $\mathfrak{g}$ is  called
   nonlinear irreducible  if the isotropy  representation  $ \mathrm{ad}_{\mathfrak{g}^0}|_{\mathfrak{g}^{-1}}$ is irreducible  and  $\mathfrak{g}^1 \neq (0)$.  An  example is  the  full prolongation of  an irreducible  linear Lie  superalgebra $\mathfrak{g}^0 \subset \mathfrak{gl}(\mathfrak{g}^{-1})$  of  finite  type with non-trivial first prolongation.  We   prove  that  a complex  Lie  superalgebra $\mathfrak{g}$
   which admits  a depth one  transitive  nonlinear irreducible grading  is a  semisimple Lie superalgebra   with  the  socle $\mathfrak{s}\otimes \Lambda(\mathbb{C}^n)$, where  $\mathfrak{s}$ is a simple Lie   superalgebra,  and  we  describe    such gradings. The graded Lie superalgebra  $\mathfrak{g}$  defines   an isotropy irreducible  homogeneous  supermanifold   $M=G/G_0$   where   $G$, $G_0$ are Lie  supergroups  respectively associated  with the Lie  superalgebras
   $\mathfrak{g}$ and $\mathfrak{g}_0 := \bigoplus_{p\geq 0} \mathfrak{g}^p$.
\end{abstract}
\maketitle
\null \vspace*{-.25in}
\bigskip
\bigskip
\section{Introduction}
\vskip0.2cm\par\noindent
 An  homogeneous  manifold $M=G/H$ of  a  connected  Lie  group $G$   with  a connected  stability subgroup $H$ is called  (isotropy) \emph{nonlinear}   if  the isotropy  representation
  $j: H \to GL(T_oM)  \simeq GL(\mathfrak{g}/\mathfrak{h})$     has  a positive  dimensional  kernel or, in other  words,  there is  no system of local  coordinates  near $o =eH \in M$   which linearizes  the  action of $H$.
	A nonlinear homogeneous manifold  $M=G/H$  is  called \emph{irreducible} (resp. \emph{primitive}) if  the  isotropy  group $j(H)$ is  irreducible
    (resp.  the  stabilizer $H$ is  a maximal connected  subgroup of $G$).
    
		All nonlinear   irreducible  homogeneous manifolds  had been  classified  by S.\ Kobayashi and  T.\ Nagano (see \cite{KN1, KN2}) and  the nonlinear primitive  homogeneous manifolds  by T.\ Ochiai (see \cite{O1}).  It turns out that any  complex nonlinear primitive   manifold  $M=G/H$  is    the  quotient of  a  complex  simple  Lie  group $G$
by  a maximal parabolic  subgroup $H= P_{\alpha}$  associated   to  a simple root $\alpha$ (the   semisimple part  of $P_{\alpha}$  is    the   semisimple  regular  subgroup of $G$   associated  with    the  Dynkin  diagram  of  $G$    with   deleted   root $\alpha$).   Any  real   nonlinear primitive homogeneous manifold  has  the  form  $M^{\sigma} = G^{\sigma}/P^{\sigma}_{\alpha}$
  where $G^{\sigma}$  is  a real  form of  $G$      defined  by an anti-linear involution of  the Lie  algebra  $\mathfrak{g}=Lie(G)$ which preserves   the  subalgebra  $\mathfrak{p}_{\alpha}=Lie(P_\alpha)$
  and $P_{\alpha}^{\sigma} $   is    the  parabolic  subgroup   associated  with the  fixed  point  set  $\mathfrak{p}_{\alpha}^{\sigma}$  of $\sigma$ in $\mathfrak{p}_{\alpha}$. The   nonlinear primitive manifold  $M= G/P_{\alpha}$  or  its  real  form  $M^{\sigma}=G^\sigma/P_\alpha^\sigma$ is  irreducible  if  and only if   the  Dynkin mark $m_{\alpha}=1$. We recall that  the Dynkin marks  are  the coordinates  of  the maximal  root  with  respect  to the basis of simple  roots.

	The  problem of  classifying the nonlinear primitive   manifolds   reduces   to the  classification  of the nonlinear primitive transitive Lie     algebras  $(L,L_0)$. Here $L$ is a Lie algebra, $L_0$ a maximal subalgebra of $L$  which does not contain  any nontrivial ideal of  $L$ (this condition is called effectivity or transitivity) and such that  the  isotropy  representation $\mathrm{ad}_{L_0} |_{L/L_0}$  has  a kernel  $L_1 \neq (0)$.  Moreover, it   corresponds  to  an irreducible  homogeneous manifold  if  the  isotropy  representation $\mathrm{ad}_{L_0} |_{L/L_0}$ is irreducible. The   subalgebra  $L_0$ defines   a natural  filtration of $L$ (see \cite{Weis})  and the study of the associated  $\bZ$-graded Lie  algebra is a crucial step  for the classification  of the primitive homogeneous manifolds.\\
\medskip\par
     In  the  present  paper,    we  deal  with the    similar problem  for supermanifolds.  Like   for the homogeneous manifolds,  a  homogeneous   supermanifold  can  be   described as a   quotient  $M=G/H$ of  a Lie  supergroup $G$    by a  Lie  supersubgroup $H$   and  the  pair   $(G,H)$  is essentially determined by     the associated pair  $(L =Lie(G), L_0 =Lie(H))$  of  Lie  superalgebras, see  e.g. \cite{Santi} for more details.
    The condition  that  $M=G/H$ is  nonlinear primitive  or irreducible   can also be also rephrased  in   terms of  the pair $(L,L_0)$ of Lie  superalgebras, see \S \ref{preliminary}. We  prove in \S \ref{sec:2} that if   $(L,L_0)$ is  a  complex  nonlinear primitive  Lie  superalgebra,  then the  Lie  superalgebra  $L$ is  semisimple  and  its  socle $\mathrm{soc}(L)$  is necessarily of the form $\mathfrak{s}^{\Lambda}=\gs\otimes\L$, where  $\mathfrak{s}$  is  a  simple Lie  superalgebra and $\Lambda=\Lambda(V^*)=\mathbb{C}1\oplus\Lambda^+$ is the algebra of 
			exterior  forms  on $V = \mathbb{C}^{n}$ for some $n\geq 0$ (we recall the the socle of a Lie superalgebra  $L$ is the  sum  of  all minimal ideals of   $L$).
    According  to  V.\ G.\ Kac in \cite{K0} (see also \cite{Ch}),  any   such Lie  superalgebra  $L$ is a    subalgebra of  the   Lie  superalgebra of  derivations of  the  socle
   $$\mathrm{der}(\mathfrak{s}^{\Lambda})=   \mathrm{der}(\mathfrak{s})\otimes \Lambda\niplus 1 \otimes W\;,$$
    where $W = \mathrm{der}(\Lambda) = \Lambda\otimes \partial_V$ is the algebra of derivations of $\L$ and $\partial_V$ is the space of  ``constant  odd  supervector  fields''. Moreover, the condition of semisimplicity for $L$ is equivalent to the fact that the restriction to $L$ of  the  natural projection
     $$\mathrm{der}(\mathfrak{s}^{\Lambda})\longrightarrow\mathrm{der}(\mathfrak{s}^{\Lambda}) / (\mathrm{der}(\mathfrak{s})\otimes \Lambda  + 1\otimes \Lambda^+ \partial_V ) \simeq \partial_V $$ is  surjective. In this paper any such semisimple Lie superalgebra $L  \subset \mathrm{der}(\mathfrak{s}^{\Lambda})$  is  called an \emph{intermediate admissible} Lie superalgebra. 
		
		Any maximal subalgebra $L_0$ of $L$ defines a filtration (see \cite{Weis, CK}) 
		$$L=L_{-d}\supset L_{-d+1}\supset\cdots\supset L_0 \supset L_1 \supset\cdots$$  and the associated  $\bZ$-graded transitive Lie  superalgebra of  depth $d(\gg)=d$:
		$$\mathfrak{g}=\mathrm{gr}(L)=L_{-d}/L_{-d+1}\oplus\cdots\oplus L_0/L_1\oplus L_1/L_2\oplus\cdots\;.$$ 
		The problem of classifying the  nonlinear  primitive Lie  superalgebras  $(L,L_0)$ where  $L$ is a semisimple  Lie   superalgebra as above is an involved problem.
    The  first  step   for  its solution is   to   describe   transitive   primitive or  irreducible  $\bZ$-graded Lie  superalgebras. In  the  case of Lie  algebras,  any finite  dimensional primitive filtered Lie  algebra  $L$ is  isomorphic  to   the  associated $\bZ$-graded Lie algebra $\gg=\mathrm{gr} (L)$ (see \cite{KN4, O1}). A similar statement is also true for the infinite-dimensional primitive Lie algebras (see \cite{SS, Gui1, MT}).
		
		For Lie  superalgebras, the situation is more  complicated  and  a filtered  Lie superalgebra  can often be obtained  by   an appropriate  deformation of the bracket of the associated  $\bZ$-graded  Lie  superalgebra (see \cite{ChK0}; see also \cite{FS} for an instance of this phenomenon in the context of supersymmetric field theories). The full classification of the infinite-dimensional primitive (filtered) Lie superalgebras   was obtained  by N.\ Cantarini, S.-J.\ Cheng  and V.\  Kac in \cite{CK, ChK}.	
		\medskip\par
      The present paper gives a full description  of the \emph{finite  dimensional nonlinear transitive irreducible $\bZ$-graded}  Lie  superalgebras $\mathfrak{g} = \bigoplus_{i=-1}^{\ell} \mathfrak{g}^i$.
          In this case, the grading  of  $\mathfrak{g}$  defines   a filtration of the form
          $\mathfrak{g}_{-1} = \mathfrak{g} \supset \mathfrak{g}_0 \supset \mathfrak{g}_1 \supset \cdots $    where
          $\mathfrak{g}_p := \bigoplus_{i \geq p} \mathfrak{g}^i $ and  the associated  pair $(L,L_0)=(\mathfrak{g}, \mathfrak{g}_0)$ is an irreducible Lie superalgebra which is nonlinear, that is with $L_1=\gg_1=\bigoplus_{i\geq 1}\gg^i\neq (0)$. The study of the filtered deformations
				of $\mathfrak{g} = \bigoplus_{i \geq -1} \mathfrak{g}^i$ will be the content of a future work. 
				
				We remark  that  the  \emph{full} prolongation of a  linear Lie  superalgebra  $\mathfrak{g} \subset  \mathfrak{gl}(V)$ which acts irreducibly on a supervector space $V=V_{\0}\oplus V_{\1}$  can be  defined  as the  {\it maximal} depth one   $\bZ$-graded transitive
          Lie   superalgebra $V\oplus\mathfrak{g} \oplus \mathfrak{g}^{(1)} \oplus \mathfrak{g}^{(2)}\oplus \cdots$   which  extends  $V \oplus \mathfrak{g}= \mathfrak{g}^{-1} \oplus \mathfrak{g}^0$. All  such full prolongations which are infinite-dimensional were  described in  \cite{K2, ChK}.
	On the other hand	the   full prolongation
 of an irreducible linear   Lie  superalgebra  $\mathfrak{g}$  with nontrivial first prolongation $\mathfrak{g}^{(1)} \neq (0)$
 is finite  dimensional if  and only if  there are no elements in $\ad_{\gg_{0}}|_{(\gg_{-1})_{\0}}$ of rank one,  see \cite{ALS, Wil}, and  the class of all  such  graded  Lie  superalgebras is  just a subclass of the
  finite dimensional nonlinear transitive irreducible graded Lie  superalgebras  described  below. This  gives, in particular, a description of all the irreducible representations of Lie  superalgebras with a nontrivial first prolongation. This description is clearly
of interest \emph{per se} but we would also like to  mention that some representations with nontrivial first prolongation had recently played an important r$\hat{\text o}$le in a
partial generalization to the case of supermanifolds of the classical result of M.\ Berger on the possible irreducible holonomy algebras of Riemannian manifolds (see \cite{Ga, Ga2}).
\medskip\par
The main results of this paper are Theorem \ref{importantcase} in \S \ref{sec:3} and Theorem \ref{main1} in \S \ref{sec:4} and
the description in \S \ref{basicgrad}, \S \ref{strangegrad} and \S \ref{cartangrad} of all the $\bZ$-gradings of depth one of the Lie superalgebras $\mathrm{der}(\mathfrak{s})$ of derivations of simple Lie superalgebra $\mathfrak{s}$. We summarize the main results in the following Theorem
\ref{collec} for the reader's convenience. To state it, we recall that 
the Lie superalgebra $\mathrm{der}(\mathfrak{s})$ always admits a canonical decomposition into the semidirect  sum   $\mathrm{der}(\mathfrak{s}) = \mathfrak{s} \niplus \mathrm{out}(\mathfrak{s})$ of the ideal  of
inner derivations (identified  with  $\mathfrak{s}$) and a  complementary subalgebra $\out(\gs)$ of outer  derivations (see \cite{K0}).
\begin{theorem}
\label{collec}
Let  $\mathfrak{g} = \bigoplus_{i=-1}^{\ell} \mathfrak{g}^i$ be a  transitive  nonlinear  irreducible  $\mathbb{Z}$-graded  Lie  superalgebra of  depth one.  
Then $\gg$ is a semisimple Lie superalgebra with socle $\mathrm{soc}(\mathfrak{g})= \mathfrak{s}^{\Lambda}=\gs\otimes\L$ where  $\mathfrak{s}$ is a uniquely  determined  simple Lie  superalgebra  and $\Lambda$ is  the algebra of exterior forms on $V=\bC^n$, for $n\geq 0$.   
The superalgebra $\mathfrak{g}$  is   a graded  subalgebra   of the  Lie  superalgebra  
\begin{align*}
\oder(\mathfrak{s}^{\Lambda}) &= \mathrm{der}(\mathfrak{s})\otimes \Lambda \niplus 1\otimes W\\
&=\mathfrak{s}^\Lambda  \niplus \mathrm{out}(\mathfrak{s}^{\L})\;,
\end{align*}
where $W=\mathrm{der}(\L)$ and $\mathrm{out}(\mathfrak{s}^{\L})=  \mathrm{out}(\mathfrak{s})\otimes \L \niplus 1 \otimes W$, with one of  the  following  gradings:
\begin{itemize}
\item[(I)] $\gs=\bigoplus\gs^p$ has a $\bZ$-grading of depth $1$ and $\L=\L^0$ and $W=W^0$ have the trivial gradings;
\item[(II)] $\gs=\gs^0$ has the trivial grading whereas $\L$ and $W$ have the gradings defined by a  grading $V=V^0\oplus V^{1}$ with  $\dim V^{1}=1$.
\end{itemize}
If $n=0$ then the socle of $\gg$ is a simple Lie superalgebra $\gs$, the grading of $\gg$ is as in (I) and $\gg$ is a $\bZ$-graded Lie superalgebra of the form $\mathfrak{g}= \mathfrak{s}\niplus F$ where $F \subset   \mathrm{out}(\mathfrak{s})$ is a nonnegatively $\bZ$-graded subalgebra of $\out(\gs)$. The relevant $\bZ$-gradings of $\out(\gs)$ and $\oder(\gs)$ are described in \S \ref{basicgrad}, \S \ref{strangegrad} and \S \ref{cartangrad}. If $n>0$ then $\gg$ is the semidirect sum $\mathfrak{g} = \mathfrak{s}^{\Lambda} \niplus  F$,  where $F=\bigoplus F^p$ is a  nonnegatively graded   subalgebra  of  the Lie  superalgebra $\mathrm{out}(s^{\Lambda})$ which is  admissible, that is    the  natural projection from $F$ to $\partial_V$ is  surjective.

Conversely any grading as in (I) or (II) defines a transitive nonlinear and irreducible grading of depth one of the Lie superalgebra $\gg=\gs^\L\niplus F$ for any nonnegatively graded subalgebra $F$ of $\out(\gs^\L)$ which is admissible.
\end{theorem}
The paper is organized   as  follows. In \S \ref{sec:2} we give the main preliminary definitions and results on nonlinear primitive Lie superalgebras and  recall the  description of the semisimple  Lie superalgebras. In \S \ref{sec:3.1}, \S \ref{sec:3.2}, \S \ref{sec:3.3} and \S \ref{sec:3.4} we give some basic results on irreducible $\bZ$-graded Lie superalgebras, describe the   gradings of the Grassmann  algebra  $\Lambda$  and  its Lie  superalgebra  $W=\mathrm{der}(\Lambda)$  of derivations and, finally, give the description of the $\bZ$-gradings of the semisimple Lie superalgebras with socle $\gs^\L=\gs\otimes\L$. Section 
\ref{weaklyouter} contains Theorem \ref{importantcase} and the description of the gradings of $\oder(\gs)$ associated with the depth one gradings of the simple Lie superalgebras $\gs$. We mention that, besides the well-known last grading of the strange Lie superalgebra $\gs=\mathfrak{spe}(n)$ described in Table \ref{primatab35},  there is another interesting grading that is of depth $d=1$ and height $\ell=2$: it is a grading on a subalgebra of the algebra $\oder(\gs)$ of derivations of the basic Lie superalgebra $\gs=\mathfrak{psl}(2|2)$ (see Table \ref{TABSL22}). We apply  these  results   in  Section \ref{sec:4} to the description in Theorem \ref{main1} of the  depth one  transitive nonlinear  irreducible  gradings of  semisimple  Lie  superalgebras and conclude with some observations.
\subsubsection*{Notations}\hfill
\vskip0.2cm\par
Given any supervector space $V=V_{\bar 0}\oplus V_{\bar 1}$, we denote by 
$$\Pi V=(\Pi V)_{\bar 0}\oplus(\Pi V)_{\bar 1}$$ the supervector space with opposite parity, that is 
\begin{equation}
\label{pcf}
(\Pi V)_{\bar 0}=V_{\bar 1}\;\;,\qquad(\Pi V)_{\bar 1}=V_{\bar 0}
\end{equation}
as (non-super) vector spaces. We set $\bC^{m|n}=\bC^{m}\oplus\Pi(\bC^{n})$. 

The tensor product $V\otimes V'$ of two supervector spaces has a natural structure of supervector space given by
$$
(V\otimes  V')_{\overline 0}=(V_{\overline 0 }\otimes V'_{\overline 0})\oplus (V_{\overline 1}\otimes V'_{\overline 1})\,,\,(V\otimes  V')_{\overline 1}=(V_{\overline 0 }\otimes V'_{\overline 1})\oplus (V_{\overline 1}\otimes V'_{\overline 0})\ .
$$
In \S \ref{basicgrad}-\S\ref{strangegrad} (and only there) we use the symbols
$\L(\bC^{m|n})$ and $S(\bC^{m|n})$ to denote the exterior and symmetric products of $V=\bC^{m|n}$ in the super sense.

We denote finite-dimensional simple Lie superalgebras according to the conventions in e.g. \cite{Serga, Leites} (some authors use different conventions, especially for classical Lie superalgebras; cf. \cite{K0, Sc}). 

The projectivization of a linear Lie superalgebra $\gg$ which contains the scalar matrices is denoted by $\mathfrak{pg}=\gg/\mathbb{C}\Id$, if it does not contain $\mathbb C\Id$ we set $\mathfrak{pg}:=\gg$.

Finally the symbol $\niplus$ stands for the semidirect sum  of Lie superalgebras. If 
$\gg=\gg_{1}\niplus\gg_{2}$ then $\gg_1$ (resp. $\gg_2$) is an ideal (resp. a subalgebra) of $\gg$.
\vskip0.2cm\par\noindent
\subsubsection*{Acknowledgments}\hfill
\vskip0.2cm\par
D.\ V.\ A.\ acknowledges support of an  RNF
grant (project n.14-50-00150) and institutional support of the University of Hradec Kr\'alov\'e. A.\ S.\ is supported by a
Marie-Curie research fellowship of the ``Istituto Nazionale di Alta
Matematica" (Italy).  We would like to thank our respective
funding agencies for their support.

The authors are grateful to the anonymous referee for
his helpful comments.
\bigskip
\bigskip
\par
\section{Nonlinear primitive Lie superalgebras}
\label{sec:2}
\subsection{Preliminary definitions and results}\hfill\vskip0.2cm\par\noindent\label{preliminary}
Let $L=L_{\0}\oplus L_{\1}$ be a finite dimensional Lie superalgebra over the field of complex numbers and $L_0$ a subalgebra of $L$. The pair $(L,L_0)$ is called a {\it transitive Lie superalgebra} if $L_0$ does not contain any nonzero ideal of $L$. A transitive Lie superalgebra is called {\it nonlinear primitive} if
\begin{itemize}
\item[(a)] $L_0$ is a maximal subalgebra of $L$ ({\it primitive});
\item[(b)] the isotropy representation $\left.\ad_{L_0}\right|_{L/L_0}$ of $L_0$ on $L/L_0$ has a nontrivial kernel, i.e., the subspace $L_1$ defined by
$$
L_1=\left\{x\in L_0\!\!\!\!\!\!\!\!\!\phantom{L^{L^{L}}}|\,[x,L]\subset L_0\right\}
$$
is nonzero ({\it nonlinear}).
\end{itemize}
A transitive Lie superalgebra $(L,L_0)$ is called {\it irreducible} if the isotropy
representation is irreducible;
it is immediate that any irreducible transitive Lie superalgebra is primitive. Two transitive Lie superalgebras $(L,L_0)$ and $(L',L_0')$ are {\it isomorphic} if there exists a Lie superalgebra isomorphism
$\varphi:L\to L'$ with $\varphi(L_0)=L_0'$. 
\vskip0.2cm\par
The following result is proved in \cite[Lemmata 1 and 2]{O1} (the proof is given in the
Lie algebra case but it extends verbatim to the superalgebra case).
\begin{lemma}
\label{inizio}
Let $(L,L_0)$ be a nonlinear primitive Lie superalgebra. Then:
\begin{itemize}
\item[(i)] $L_0\neq L_1$;
\item[(ii)] $[L_0,L_1]\subset L_1$;
\item[(iii)] $[L,L_1]\subset L_0$.
\end{itemize}
Furthermore:
\begin{itemize}
\item[(iv)]  if $J$ is a nonzero ideal of $L$, then $J+L_0=L$;
\item[(v)] if $J_1$ and $J_2$ are ideals of $L$ such that $J_1\cap L_0\neq (0)$ and $J_2\neq (0)$, then 
$[J_1,J_2]\neq (0)$;
\item[(vi)] if $J$ is a nonzero ideal of $L$, then $J\cap L_0\neq (0)$;
\item[(vii)] two nonzero ideals of $L$ are never in direct sum.
\end{itemize}
\end{lemma}
Recall that a Lie superalgebra is called {\it semisimple} if its radical is zero, equivalently if it does not contain any nonzero abelian ideal.
\begin{corollary}
\label{iniziobis}
Let  $(L,L_0)$ be a  nonlinear primitive  Lie   superalgebra.  Then $L$ is  a  semisimple Lie   superalgebra and  it is not  the direct  sum  of  two nonzero ideals.
\end{corollary}
In the Lie algebra case this immediately yields that $L$ is simple, $L_1$ abelian and $L_2=(0)$ (see \cite[Lemma 3]{O1}). 
The corresponding statement is not true in the Lie superalgebra case and one has
to separately consider simple and semisimple Lie superalgebras. 

Semisimple Lie superalgebras are described in \cite{K0} (see also \cite[Section 7]{Ch}).
We recall here only the facts that we need and refer to those texts for more details.
\subsection{Semisimple Lie superalgebras}\label{desofsemi}\hfill\vskip0.2cm\par\noindent
Let $n$ be a nonnegative integer and $V=\mathbb{C}^{n}$. We denote by $\L=\L(n)=\L(V^*)$ the Grassmann algebra of $V^*$ with its natural parity decomposition $\Lambda=\Lambda_{\0} \oplus\Lambda_{\1}$ and by
$$W=W(n)=\mathrm{der}(\L)=\Lambda\otimes\partial_V$$
the \LSA \, of  supervector fields on the $n$-dimensional purely odd linear supermanifold. 
Here $\partial_V=\{\partial_{\xi}\,|\,\xi\in V^*\}\simeq V$ is the \LSA\,of  constant  supervector  fields. 
There exists a direct sum decomposition 
\beq
\label{splitvect}
W=\partial_V\oplus \Lambda^+\partial_V\;,
\eeq 
where $\Lambda^{+}$ is the (not unital) subalgebra of $\Lambda$ generated by $V^*$.
\begin{definition}\label{tr} 
A subalgebra of $W$ is called {\it admissible} if it projects surjectively to $\partial_V$ with respect to the decomposition \eqref{splitvect}.
\end{definition}

Let $\gs_1,\ldots,\gs_N$ be simple \LSA s and $n_1,\ldots,n_N$ nonnegative integers. We associate two Lie  superalgebras
\beq
\label{algmin}
\gs^{\L}=\bigoplus\gs_i\otimes \Lambda(n_i)
\eeq
and 
\begin{align*}
\gs_{\max}=\oder(\gs^{\L})&=\bigoplus\oder(\gs_i\otimes \Lambda(n_i))\\
&=\bigoplus(\oder(\gs_i)\otimes\Lambda(n_i)\niplus \textbf{1}\otimes W(n_i)) 
\end{align*}
and consider the natural projections
\begin{multline}
\label{bigpro}
\pi_i: \gs_{\max} \to \gs_{\max}/\Ker\pi_i\simeq W(n_i)\qquad\text{with}\\
\Ker\pi_{i}=\bigoplus_{j\neq i}\oder(\gs_j\otimes\Lambda(n_j))\oplus(\oder(\gs_i)\otimes\Lambda(n_i))\;.\,
\end{multline}
Note that $\gs^{\L}$ is a subalgebra of $\gs_{\max}$ in a natural way. We call any subalgebra 
$\gs^{\L}\subset \gg \subset \gs_{\max}$ an {\it intermediate subalgebra}.
\begin{definition}\label{semisimple}
An intermediate subalgebra $\gg$ is  {\it admissible}
if its projection $\pi_i(\gg)$ is an admissible subalgebra of $W(n_i)$ for all $i=1,\ldots, N$.
\end{definition}
Recall that the {\it socle} $\soc(\gg)$ of a Lie superalgebra $\gg$ is usually defined as the sum of all minimal ideals of $\gg$.
\begin{theorem}\cite{K0}\label{block}
An intermediate subalgebra $\gg$ is semisimple if and only if it is admissible and any semisimple Lie superalgebra is an admissible intermediate subalgebra,
for some $\gs_1,\ldots,\gs_N$ and $n_1,\ldots,n_N$. 
Moreover if $\gg$ is semisimple:
\begin{itemize} 
\item[(i)] each $\gs_i\otimes\Lambda(n_i)$ is a minimal ideal of $\gg$ and
$\soc(\gg)=\gs^{\Lambda}$;
\item[(ii)] $\oder(\gg)$ coincides with the normalizer $N_{\gs_{\max}}(\gg)$ of $\gg$ in $\gs_{\max}$.
\end{itemize}
\end{theorem}
The socle of a semisimple Lie superalgebra $\gg$ is a {\it characteristic ideal}, i.e. it is stable under all automorphisms and derivations of $\gg$ (it can be proved directly for the automorphisms and the even derivations, one needs point (ii) of Theorem \ref{block} for the odd derivations). This leads to the following useful observations.
\begin{remark}\label{external0}
The $\gs_1,\ldots,\gs_N$ and $n_1,\ldots,n_N$ of Theorem \ref{block} are uniquely determined by $\gg$. 
Indeed, since each $\gs_i\otimes\Lambda(n_i)$ is a minimal ideal, if
$$\soc(\gg)=\bigoplus_{i=1}^{N}\gs_{i}\otimes\Lambda(n_i)=\bigoplus_{i=1}^{M}\gs'_i\otimes\Lambda(m_i)$$
then $N=M$ and
$
\gs_{i}\otimes\Lambda(n_i)=\gs'_i\otimes\Lambda(m_i)
$
up to a reordering of the indices. The claim follows from the fact that 
the radical of $\gs_i\otimes\Lambda(n_i)$ is $\gs_i\otimes\Lambda^{+}(n_i)$ and 
$\gs_i=\gs_i\otimes\Lambda(n_i)/\gs_i\otimes\Lambda^{+}(n_i)$.
\end{remark}
\begin{remark}\label{external}
Any isomorphism $\phi:\gg\to\gg'$ of semisimple Lie superalgebras is the restriction $\phi=\varphi|_{\gg}$ of an automorphism $\varphi$ of $\gs_{\max}$ with $\varphi(\gg)=\gg'$. Indeed $\soc(\gg)=\soc(\gg')=\gs^\L$ and $\phi(\gs^\L)=\gs^\L$ by Remark \ref{external0}; the claim follows then from the fact that $\phi$ is completely determined by the action on $\gs^\L$ and this action, in turn, always extends to a unique automorphism of $\gs_{\max}=\oder(\gs^{\L})$.
\end{remark}
We end this section with the following.
\begin{proposition}\label{graduato0}
Let $(L,L_0)$ be a nonlinear primitive Lie superalgebra. Then
$L$ is semisimple with socle of the form $\gs^{\L}=\gs\otimes \Lambda(n)$ for some unique simple Lie superalgebra $\gs$ and nonnegative integer $n$.
In particular  $L$ is  an admissible intermediate subalgebra  of  the  Lie superalgebra
$\gs_{\max}=\mathrm{der}(\mathfrak{s})\otimes \Lambda\niplus 1 \otimes W(n)$.
\end{proposition} 
\begin{proof}
Semisimplicity is direct from Corollary \ref{iniziobis}. From Theorem \ref{block}, $L$ is an intermediate admissible subalgebra and we assume $N\geq 2$, by contradiction. Then $s_1\otimes\Lambda(n_1)$ and $\gs_2\otimes\Lambda(n_2)$ are two ideals of $L$ in direct sum, an absurd by Lemma \ref{inizio}. Hence $\gs^{\L}=\gs\otimes\Lambda(n)$ and unicity follows from Remark \ref{external0}.
\end{proof}
\medskip\par
\section{Filtrations and $\bZ$-gradings of Lie superalgebras}\label{sec:3}
From now on we will restrict to \emph{irreducible} Lie superalgebras.
\vskip0.3cm\par\noindent
\subsection{Preliminary results}\label{sec:3.1}\hfill\vskip0.3cm\par\noindent
Let $(L,L_0)$ be an {\it irreducible} transitive Lie superalgebra, i.e. the representation of $L_0$ on $L/L_0$ is irreducible. Following 
\cite{Weis} we set $L_{-p}=L$ for every $p\geq 1$ and define subspaces $L_p$ of $L$ inductively:
$$
L_p=\left\{x\in L_{p-1}\!\!\!\phantom{{l^{l}}^{l}}|\,[x,L]\subset L_{p-1}\right\}\ .
$$ 
The collection $\{L_p\}_{p\geq -1}$ is a {\it filtration} of $L$, 
a chain of linear subspaces 
$$L_{p}\subset L\,\,\,,\,\,\,\,\,\,L_{p}=L_{p}\cap L_{\bar 0}\oplus L_{p}\cap L_{\bar 1}\,\,\,,$$ 
which satisfies
\[\begin{gathered}
L=\bigcup_{p}L_{p},\\
0=\bigcap_{p}L_{p},\\
\end{gathered}
\]
and, for all $p,q\in\mathbb Z$,
\[\begin{gathered}
L_{p}\supset L_{p+1},\\
[L_{p},L_{q}]\subset L_{p+q}.\\
\end{gathered}
\]
The \emph{associated $\mathbb Z$-graded Lie superalgebra} $\mathfrak g=\operatorname{gr}(L)$ is defined by
\begin{equation}
\label{eq:assocgraded}
\begin{aligned}
\gg&=\bigoplus_{p\in\bZ}\mathfrak \gg^p\,,\,\,\,\,
\mathfrak \gg^p=L_{p}/L_{p+1}\,,
\end{aligned}
\end{equation}
with the induced Lie bracket and parity decomposition. 

Recall that a $\bZ$-graded Lie superalgebra $\gg=\bigoplus\gg^p$ is called of {\it depth $d(\gg)=d$} if $\gg^{p}=(0)$ for all $p<-d$ and $\gg^{-d}\neq (0)$. The Lie superalgebra \eqref{eq:assocgraded} has the following basic properties:
\begin{itemize}
\item[(a)] it is {\it transitive} i.e. if $x\in \gg^p$, $p\geq 0$, satisfies $[x, \gg^{-1}]=(0)$ then
$x=0$;
\item[(b)] it is of {\it depth} $d=1$;
\item[(c)] it is {\it irreducible} i.e. the adjoint representation of $\gg^0$ on $\gg^{-1}$ is irreducible.
\end{itemize}
A $\bZ$-graded Lie superalgebra $\gg=\bigoplus\gg^p$ satisfying (a)-(c) is called a 
{\it transitive irreducible $\bZ$-graded Lie superalgebra of depth $1$}.
Note that the corresponding zero-degree part is an irreducible linear Lie superalgebra $\gg^0\subset \ggl(\gg^{-1})$. 

From (i) of Lemma \ref{inizio} the Lie superalgebra $\gg$ associated with a \emph{nonlinear} $(L,L_0)$ has also the additional property
\vskip0.25cm
\begin{itemize}
\item[(d)] it is {\it nonlinear} i.e. $\gg^1\neq (0)$
\end{itemize}
\vskip0.15cm
and this fact motivates the following.
\begin{definition}\label{def1}
A $\bZ$-graded Lie superalgebra $\gg=\bigoplus\gg^p$ satisfying (a)-(d) is called a {\it  transitive nonlinear irreducible $\bZ$-graded Lie superalgebra of depth $1$}.
\end{definition}
There exists a kind of inverse construction $\gg\mapsto (L,L_0)$.
\begin{proposition}\label{graduato}
If $\gg=\bigoplus\gg^p$ is a transitive nonlinear irreducible $\bZ$-graded Lie superalgebra of depth $1$ 
then $\gg$ is semisimple with socle $\gs^{\L}=\gs\otimes \Lambda(n)$ for some unique
simple Lie superalgebra $\gs$ and nonnegative integer $n$. Moreover
\beq
\label{converse}
\gg=\gg_{-1}\supset\gg_{0}\supset\gg_{1}\supset\cdots\supset\gg_{p}\supset\cdots\;,\qquad\gg_{p}=\bigoplus_{i\geq p}\gg^i\;,
\eeq
is the natural filtration on $\gg$ associated with the nonlinear irreducible transitive Lie superalgebra $(L,L_0)=(\gg,\gg_0=\bigoplus_{i\geq 0}\gg^i)$.
\end{proposition} 
\begin{proof}
By a direct check, the pair $(\gg,\gg_0)$ is a nonlinear irreducible transitive Lie superalgebra with associated filtration \eqref{converse} and the rest follows from Proposition \ref{graduato0}. 
\end{proof}
\medskip\par
By Proposition \ref{graduato} the first natural class of nonlinear irreducible transitive Lie superalgebras to be considered is given by the filtered Lie superalgebras directly associated with the $\bZ$-graded Lie superalgebras of Definition \ref{def1}; the main aim of the paper is the study of
these latter. 

We now need to recall the description of $\bZ$-gradings of superalgebras.
We deal only with the facts that we need and refer to e.g. \cite{IO} for more
details.
\vskip0.3cm\par\noindent
\subsection{Gradings of a superalgebra}\label{sec:3.2}\hfill\vskip0.3cm\par\noindent
Let $A$ be a (purely even) algebra. A {\it $\bZ$-grading of $A$} (shortly a {\it grading}) is a direct sum decomposition of $A$ in linear subspaces
\beq
\label{gradingalg}
A=\bigoplus_{p\in\bZ} A^{p}\;\qquad\text{with}\;\;\; A^p \cdot A^q
\subset A^{p+q}\;.
\eeq
The grading is called of {\it depth $d$} if  $A^{-d} \neq (0)$ and $A^{-p}=(0)$ for all $p>d$.
In other words, a grading is the eigenspace decomposition for a semisimple derivation $D\in\oder(A)$, $D|_{A^p}=p\Id$ with integer eigenvalues.
The operator $D$ is called the {\it grading operator}. If $A=A_{\0} \oplus A_{\1}$ is a superalgebra, we assume that
the grading is of the form $A^p=A^p_{\0}\oplus A^p_{\1}$ i.e. $D\in\oder_{\0}(A)$ is even.
\smallskip\par
Let $\Aut(A)$ be the Lie group of all automorphisms of $A$. The description of $\bZ$-gradings of $A$ up to automorphisms corresponds to the description of the grading operators up to conjugacy by $\Aut(A)$.
\begin{definition}
A subalgebra $\gt$ of $\oder_{\0}(A)$  is called {\it toric} if it is abelian and 
any $D\in\gt$ is semisimple. 
A maximal toric subalgebra
$\gt\subset \oder_{\0}(A)$ is called a {\it maximal torus of $A$}. 
\end{definition}
Since $\oder_{\0}(A)$ is algebraic (it is the tangent Lie algebra of $\Aut(A)$, an algebraic subgroup of the group $\GL_{\0}(A)$ of all even invertible transformations), then all maximal tori of $A$ are conjugate by $\Aut^{\circ}(A)$
(\cite{OV}). 

It follows that the description of $\bZ$-gradings of a superalgebra $A$ up to
automorphisms reduces to find a maximal torus of $A$ and
its grading operators. 
\vskip0.3cm\par\noindent
\subsection{Gradings of $\L=\L(V^*)$ and $W=\oder(\L)$}\label{sec:3.3}\hfill\vskip0.3cm\par\noindent\label{gradW}
 Let $V=\bC^n$. We describe all the gradings of the Grassmann algebra $\L=\L(U)$, $U=V^*$ and the Lie superalgebra $W=\oder(\L)=\Lambda\otimes \partial_V\simeq\L\otimes V$ of supervector fields, where  $\partial_V \simeq V$ is  the  space of  constant supervector  fields. 
Let $U=U^{k_1}\oplus\cdots\oplus U^{k_m}$ be  a  grading  of  $U$, $\dim U^{k_i}=n_i$, defined by  the grading  operator
$$
D=D_U=k_1\Id_{U^{k_1}}+\cdots+k_m\Id_{U^{k_m}}
$$
(in this case just a semisimple linear endomorphism with integer eigenvalues). 
\begin{definition}
The spectrum $\vec{k}:=(k_1,\ldots,k_m)$ of $D$ is called the {\it type} of the grading.
\end{definition}
We always assume that $n_i\neq 0$ and $k_1<\cdots<k_m$. The operator  $D$ naturally  acts   as a grading operator on  $V$,  $\Lambda$  and $W$
and  defines   the  grading
$$V =  V^{-k_1} \oplus \cdots \oplus V^{-k_m }\;,\qquad V^{-k_i} = (U^{k_i})^*\;, $$
of $V$ and the gradings $\L=\bigoplus\L^p$ and $W=\bigoplus W^p$ of $\L$ and, respectively, $W$ given by 
\beq
\begin{split}
\L^p&=\!\!\!\bigoplus_{p_1k_1+\cdots+p_mk_m=p}\!\!\!\L^{p_1}(U^{k_1})\cdot\L^{p_2}(U^{k_2})\cdots \L^{p_m}(U^{k_m})\;,\\
W^q&=\!\!\!\bigoplus_{p_i-k_i=q}\L^{p_i}\partial_{V^{-k_i}}\;.\\
\end{split}
\eeq
We denote the corresponding depths by $d(\L)$ and $d(W)$.
By a little abuse of notation, we  also say  that  these gradings are of \emph{type $\vec k$}.
The  grading of type $\vec k =(1)$, defined    by   the identity  endomorphism $D =\Id$,  is usually called  the {\it principal grading};
in this case the degree $\deg a$ of  $a \in \L$  coincides  with  the  degree of $a$ as an exterior form. 
\begin{proposition}\label{torostand00} The space $\gt_{\rm W}$ of diagonal matrices in $\ggl(V)\subset W$
is a maximal torus of $W$ and all maximal tori are conjugated by $\Aut(\Lambda)$. Any grading of  $\Lambda$   and $W$  is  conjugated by $\Aut(\Lambda)$ to  a  grading  of   type  $\vec{k}$.
\end{proposition}
\begin{proof}
First note that the center $Z(W)$ of $W$ is trivial and that $\oder(W)=W$ (one directly checks the statement if $n=1$ whereas, for $n\geq 2$, $W(n)$ is simple and \cite[Prop. 5.1.2]{K0} applies). The space $\gt_{\rm W}$ is a subspace of $\oder_{\0}(\L)$; it is a toric subalgebra, since
the action of any of its elements on $\L$ is semisimple. One can then apply e.g. \cite[Lemma 1]{K1} to the principal grading of $W$ and get that $\gt_W$ 
is a maximal torus of $W$ (hence of $\L$ too, since $\oder_{\0}(\L)=\oder_{\0}(W)$). 
All maximal tori are conjugated, by the discussion at the end of Section 3.2.

The last claim follows from the fact that the collection of the supervector fields $H_\alpha= \xi^\a\partial_{\xi^\a}$, $\alpha=1,\ldots,n$, is a basis of $\gt_{\rm W}$ and the grading operators in $\gt_{\rm W}$ are the linear combinations of the $H_\a$'s with integer coefficients.
%
\end{proof}
Since $\deg 1=0$, the depth $d(\L)\geq 0$ and is given by
$$
d(\L) = \begin{cases} 0\;\;\text{if}\;\; k_1\geq 0,\\
                       \sum_{k_i <0}\, \, | k_i| n_i\;\;\text{if}\;\;  k_1 <0.
\end{cases}
$$
The  depth  $d(W) = k_m + d(\L)\geq  k_m $. This implies the following.
\begin{lem}
\label{gradL1}
\label{gradL0}
\label{lateron}
i) A grading  of  $\Lambda$  of type  $\vec{k}$ has depth $d(\L) =1$ if and only if $\vec k = (-1,k_2, \ldots, k_m)$
and  $\dim U^{-1} =1$, i.e. $U^{-1} = \mathbb{C}\xi$.
In this case  $d(W) = k_m +1$.
\\
ii) A grading of $\L$ of type $\vec k$ has depth $d(\L)=0$ if and only if $k_1 \geq 0$. 
Moreover $d(W)=0$ if and only if $\vec k =(0)$ (trivial grading)
whereas  $d(W) =1$ if and only if $\vec k = (0,1)$ or $\vec k=(1)$.
\end{lem}
\vskip0.3cm\par\noindent
\subsection{Gradings of a semisimple Lie superalgebra with socle $\gs^\L=\gs\otimes\L$}\label{sec:3.4}\hfill\vskip0.3cm
The following result reduces the description of graded semisimple Lie superalgebras
$\gg$ with socle $\gs^\L=\gs\otimes\L$, where $\gs$ is a simple Lie superalgebra, to the description of graded admissible subalgebras
$\gs^\L\subset\gg=\gs_{\max}$
of the Lie superalgebra $\gs_{\max}=\oder(\gs)\otimes\L\niplus\textbf{1}\otimes W$ with a fixed grading.
To state the result we note that any grading of $\gs$ induces   a natural grading on its algebra $\oder(\gs)$ of derivations.
\begin{proposition}\label{grad}\label{torostand}
Any grading of a semisimple Lie superalgebra $\gg$ with socle $\gs\otimes\L$ is, up to automorphisms in $\Aut(\gs_{\max})$, induced by the grading of $\gs_{\max}$
$$
\gs_{\max}=\bigoplus\gs^p_{\max}\;,\qquad\gs^p_{\max}=\bigoplus_{i+j=p} (\mathrm{der}^i(\gs)\otimes \L^j)\oplus W^p\;,
$$
generated by a grading of $\gs$ and the gradings of $\L$ and $W$ associated with a given grading of $U$of type $\vec k=(k_1,\ldots,k_m)$.
\end{proposition}
\begin{proof}
We first prove that $\gt=\gt_{\gs}\oplus\gt_{\rm W}$ is a maximal torus of $\gs_{\max}$, where $\gt_{\gs}$ is a maximal torus of $\gs$ (that is a maximal toric subalgebra of $\oder(\gs)$) and $\gt_{\rm W}$ is the space of diagonal matrices in $W$.

It is clearly toric, as
$[\gt_\gs,\gt_{\rm W}]=(0)$ 
and the action of any $x\in\gt$ is semisimple on $\gs^{\L}$ and hence on $\gs_{\max}=\oder(\gs^\L)$. Now $\gt_{\rm W}$ is a maximal torus of $W$ by Proposition \ref{torostand00}
and $\gt_{\gs}$ is maximal toric in $\oder(\gs)\otimes\Lambda$. Indeed if $x\in \oder(\gs)\otimes\Lambda$ is an even element with $[x,\gt_{\gs}]=(0)$ then $x=x^o+x^+\in\oder(\gs)\oplus\oder(\gs)\otimes\L^+$ with $[x^o,\gt_\gs]=[x^+,\gt_\gs]=(0)$ and its action on $\gs\otimes\L^{n}\simeq \gs$ reduces to the natural action of $x^o$. It follows that if $x$ is semisimple then $x^o$ is semisimple, $x^o\in\gt_{\gs}$ and $x^+=0$.

This yields
that $\gt$ is a maximal torus of $\gs_{\max}$ and that any grading operator $D\in\gt$ of $\gs_{\max}$ decomposes in $D=D_{\gs}+D_{\rm W}$ where $D_{\gs}\in\gt_{\gs}$ (resp. $D_{\rm W}\in\gt_{\rm W}$) is a grading operator of $\gs$ (resp. a grading of type $\vec k=(k_1,\ldots,k_m)$).

There is a natural $1$-$1$ correspondence between the grading operators of $\gg$ and those of $\gs_{\max}$ preserving $\gg$. First of all note that 
the grading operators of $\gs_{\max}$ and those of $\gs^{\L}$ are in a natural $1$-$1$ correspondence,
as  $\oder(\gs_{\max})=\gs_{\max}=\oder(\gs^{\L})$. Finally any grading operator of $\gg$ satisfies $$D\in\oder_{\0}(\gg)=N_{\gs_{\max}}(\gg)_{\0}\subset (\gs_{\max})_{\0}$$ by Theorem \ref{block}
and $D|_{\gs^{\L}}:\gs^\L\to\gs^\L$ is a grading operator of $\gs^{\L}$ which extends
to a unique grading operator of $\gs_{\max}$.
\end{proof}
\vskip0.3cm\par\noindent
\subsection{Gradings of the Lie superalgebra $\oder(\gs)$ with $\gs$ simple}\label{weaklyouter}\hfill\vskip0.3cm\par\noindent
The main aim of this section is to study in detail the case where the socle is a simple Lie superalgebra. In other words we have $n=0$, $V=(0)$, $\L=1$ and $W=(0)$.
An important r$\hat{\text o}$le is played by the gradings of the Lie superalgebra $\oder(\gs)$ of derivations of a simple Lie superalgebra $\gs=\gs_{\bar0}\oplus\gs_{\bar1}$.
\vskip0.3cm\par\noindent
\subsubsection{Preliminary results}\hfill\vskip0.3cm\par\noindent
Finite-dimensional simple Lie superalgebras are classified in \cite{K0} and split into two main families:
{\it classical} superalgebras, for which the adjoint action of $\gs_{\bar 0}$ on $\gs_{\bar1}$ is completely reducible, and
{\it Cartan} superalgebras $W(n)$ (for $n\geq 3$), $S(n)$ (for $n\geq 4$), $\widetilde S(n)$ (for $n\geq 4$ and even), $H(n)$ (for $n\geq 5$), that is finite-dimensional superalgebras analogue to simple Lie algebras of vector fields. 
\begin{remark}
The simple Lie superalgebras $W(2)$, $S(3)$, $\widetilde{S}(2)$ and $H(4)$ are isomorphic to the classical superalgebras $\mathfrak{sl}(1|2)\simeq \mathfrak{osp}(2|2)$, $\mathfrak{spe}(3)$, $\mathfrak{osp}(1|2)$ and $\mathfrak{psl}(2|2)$, respectively. In our conventions, they are not Cartan.
\end{remark}
Classical superalgebras in turn split into the so-called {\it basic} superalgebras $\mathfrak{sl}(m+1|n+1)$ (for $m<n$), $\mathfrak{psl}(n+1|n+1)$ (for $n\geq 1$), $\mathfrak{osp}(2m+1|2n)$ (for $n\geq 1$), $\mathfrak{osp}(2|2n-2)$ (for $n\geq 3$),
$\mathfrak{osp}(2m|2n)$ (for $m\geq 2$, $n\geq 1$), $\mathfrak{osp}(4|2;\alpha)$ (for $\alpha\neq 0, \pm 1, -2,-\frac{1}{2}$), $\mathfrak{ab}(3)$, $\mathfrak{ag}(2)$, 
for which 
there exists a non-degenerate even invariant supersymmetric bilinear form $B:\gs\otimes\gs\rightarrow\bC$, and two {\it strange} families $\mathfrak{spe}(n)$ (for $n\geq 3$) and $\mathfrak{psq}(n)$ (for $n\geq 3$). The form $B$ is unique up to constant and it coincides with the Killing form of $\gs$, except for the cases $\mathfrak{psl}(n+1|n+1)$, $\mathfrak{osp}(2m+2|2m)$ and $\mathfrak{osp}(4|2;\alpha)$. The continuous family $\mathfrak{osp}(4|2;\alpha)$ consists of deformations of $\mathfrak{osp}(4|2)$ and two Lie superalgebras $\mathfrak{osp}(4|2;\alpha)$ and $\mathfrak{osp}(4|2;\alpha')$ are isomorphic if and only if $\alpha$ and $\alpha'$ lie on the same orbit under the action
of the permutation group $\mathfrak{S}_3$ generated by $\alpha\mapsto\alpha^{-1}$ and $\alpha\mapsto(-1-\alpha)$.

Table \ref{Kactab} gives the description of the even Lie subalgebra $\gs_{\overline{0}}$ of $\gs$ and its representation on $\Pi(\gs_{\overline{1}})$ (therein $\mathbb S$ denotes the spin module of $\mathfrak{spin}(7)$). 
\vskip0.2cm
{\small
\begin{table}[H]
\begin{centering}
\makebox[\textwidth]{%
\begin{tabular}{|c|c|c|}
\hline
$\gs$ & $\gs_{\overline{0}}$ & $\Pi(\gs_{\overline{1}})$\\
\hline
\hline
$\begin{gathered}\mathfrak{sl}(m+1|n+1)^{\phantom{T}}\\ m<n\end{gathered}$ & $\mathfrak{sl}(m+1)\oplus\mathfrak{sl}(n+1)\oplus Z$ & $\bC^{m+1}\otimes(\bC^{n+1})^{*}\oplus(\bC^{m+1})^{*}\otimes\bC^{n+1}$ \\
\hline
$\begin{gathered}\mathfrak{psl}(n+1|n+1)^{\phantom{T}}\\ n\geq 1 \end{gathered}$ & $\mathfrak{sl}(n+1)\oplus\mathfrak{sl}(n+1)$ & $\bC^{n+1}\otimes(\bC^{n+1})^{*}\oplus(\bC^{n+1})^{*}\otimes\bC^{n+1}$\\
\hline
$\begin{gathered}\mathfrak{osp}(2m+1|2n)^{\phantom{T}}\\ n\geq 1 \end{gathered}$ & $\mathfrak{so}(2m+1)\oplus\mathfrak{sp}(2n)$ & $\bC^{2m+1}\otimes\bC^{2n}$\\
\hline
$\begin{gathered}\mathfrak{osp}(2|2n-2)^{\phantom{T}}\\ n\geq 3 \end{gathered}$ & $\mathfrak{so}(2)\oplus\mathfrak{sp}(2n-2)$ & $\bC^{2}\otimes\bC^{2n-2}$\\
\hline
$\begin{gathered}\mathfrak{osp}(2m|2n)^{\phantom{T}}\\ m\geq 2, n\geq 1 \end{gathered}$ & $\mathfrak{so}(2m)\oplus\mathfrak{sp}(2n)$ & $\bC^{2m}\otimes\bC^{2n}$\\
\hline
$\begin{gathered}\mathfrak{osp}(4|2;\alpha)^{\phantom{T}}\\ \alpha\neq 0,\pm 1,-2,-1/2 \end{gathered}$ & $\mathfrak{sl}(2)\oplus\mathfrak{sl}(2)\oplus\mathfrak{sl}(2)$ & $\bC^{2}\otimes\bC^{2}\otimes\bC^2$\\
\hline
$\mathfrak{ab}(3)^{\phantom{T}}$ & $\mathfrak{spin}(7)\oplus\mathfrak{sl}(2)$ & $\mathbb{S}\otimes\bC^{2^{\phantom{T}}}$ \\
\hline
$\mathfrak{ag}(2)^{\phantom{T}}$ & $\mathrm{G}_2\oplus \mathfrak{sl}(2)$ & $\mathbb{C}^{7^{\phantom{T}}}\otimes\bC^2$\\
\hline 
$\begin{gathered}\mathfrak{spe}(n)^{\phantom{T}}\\ n\geq 3\end{gathered}$ & $\sl(n)$ & $S^2(\mathbb{C}^{n})\oplus \L^2((\mathbb{C}^{n})^*)$\\
\hline 
$\begin{gathered}\mathfrak{psq}(n)^{\phantom{T}}\\ n\geq 3\end{gathered}$ & $\sl(n)$ & $\ad(\sl(n))$\\
\hline 
\end{tabular}}
\end{centering}
\caption[]{\label{Kactab} Classical superalgebras $\gs=\gs_{\overline 0}\oplus \gs_{\overline 1}$.} \vskip14pt
\end{table}}
\vskip-0.4cm\par
A direct inspection of \cite[Prop. 5.1.2]{K0} and its proof implies the following. 
\begin{proposition}
\label{utile}
Let $\gs$ be a simple Lie superalgebra. Then $\oder(\gs)$ admits a semidirect decomposition $\oder(\gs)=\gs\niplus\out(\gs)$ for a
subalgebra $\out(\gs)$ of outer derivations and a maximal torus of $\gs$ (that is a maximal toric subalgebra of $\oder(\gs)$) is conjugated to a maximal torus of the form $\gt_\gs=\gh\oplus\gt_o$ for
a maximal toric subalgebra $\gh$ in $\gs_{\0}$ and a maximal toric subalgebra $\gt_o$ in $\out_{\0}(\gs)$. 
In particular there is a natural $1$-$1$ correspondence between gradings of $\gs$, gradings of $\oder(\gs)$ and grading operators $D\in\gt_{\gs}$. 
Moreover $[\gh,\out(\gs)]=(0)$ in all cases and $\gt_o$ has
dimension one if $\gs=\mathfrak{psl}(n+1|n+1)$, $\mathfrak{spe}(n)$, $S(n)$, $H(n)$ and it is trivial for all other simple Lie superalgebras. 
\end{proposition}
We note that the subalgebra $\out(\gs)$ is always stable under any grading operator $D\in\gt_\gs$ and hence graded. The cases where $\out(\gs)$ is nontrivial are summarized in Table \ref{tableouter}, for the reader's convenience (therein we denote by $T$ a two-dimensional solvable Lie superalgebra of dimension $(1,1)$ if $n$ is odd, respectively of dimension (2,0) is $n$ is even). 

{\small
\begin{table}[H]
\begin{centering}
\makebox[\textwidth]{%
\begin{tabular}{|c|c|c|c|c|c|c|}
\hline
$\gs$ & $\begin{gathered}\mathfrak{psl}(n+1|n+1)^{\phantom{c^C}}\\ n\geq 2\end{gathered}$ & $\mathfrak{psl}(2|2)$ & $\begin{gathered} S(n)^{\phantom{c^C}}\\ n\geq 4\end{gathered}$ & $\begin{gathered}\mathfrak{spe}(n)^{\phantom{c^C}}\\ n\geq 3\end{gathered}$ & $\begin{gathered}\mathfrak{psq}(n)^{\phantom{c^C}}\\ n\geq 3\end{gathered}$ & $\begin{gathered} H(n)^{\phantom{c^C}}\\ n\geq 5\end{gathered}$\\
\hline
\hline
&&&&\\[-3mm]
$\out(\gs)$&
$\bC$
& $\mathfrak{sl}(2)$ & $\bC$ & $\bC$ & $\Pi(\bC)$ & $T$
\\
\hline
\end{tabular}}
\end{centering}
\caption[]{\label{tableouter} The algebras of outer derivations of simple Lie superalgebras.} \vskip14pt
\end{table}}
\vskip-0.5cm\par\noindent
Recall that a representation of a Lie superalgebra is called {\it irreducible of $G$-type}
if it does not admit any nontrivial submodule whether or not $\bZ_2$-graded, see \cite{BL}.
The following is a basic but useful result.
\begin{lemma}\label{lemmetto}
Let $\gs$ be a simple Lie superalgebra. If $\gs=\bigoplus\gs^p$ is a grading of depth one and $\oder(\gs)=\bigoplus \oder^p(\gs)$ the associated grading of $\oder(\gs)$, then:
\begin{itemize}
\item[i)] if $x\in\oder^p(\gs)$, $p\geq 0$, satisfies $[x,\gs^{-1}]=(0)$
then $x=0$;
\item[ii)] $\gs^0$ and $\gs^1$ are nonzero;
\item[iii)] the adjoint action of $\gs^{0}$ on $\gs^{-1}$ is irreducible of $G$-type;
\item[iv)] the depth $d(\oder(\gs))\geq d(\gs)$ and if $D\in\gh$ then $d(\oder(\gs))=d(\gs)$;
\item[v)] if $x\in\gs^{-1}$ satisfies $[\gs^{0},x]=(0)$ then $x=0$.
\end{itemize} 
\end{lemma}
\begin{proof}
(i) As $\gs^{-1}=\gs^{-1}_{\0}\oplus\gs^{-1}_{\1}$ is a subspace of $\gs$, it is enough to consider an homogeneous $x$. If $x\in\gs^p$, the condition $[x,\gs^{-1}]=(0)$  implies that the ideal $<x>$ generated by $x$ in $\gs$ is nonnegatively graded, hence trivial and $x=0$. 

If $x\in\out^p(\gs)$ satisfies $[x,\gs^{-1}]=(0)$ then for all $y\in\gs^0$ and $z\in\gs^{-1}$ one has $[y,z]\in\gs^{-1}$ and
$
[[x,y],z]=[x,[y,z]]\pm[[x,z],y]=0
$.
It follows that $[x,y]\in\gs^p$ is zero and that $[x,\gs^{p}]=(0)$ for any $p\geq 0$, by a simple induction process. Summarizing the adjoint action of $x$ on $\gs$ is trivial and $x=0$.
\vskip0.05cm\par
(ii) First of all $\gs^0\neq (0)$ otherwise $\gs^p=(0)$ for any $p\geq 0$ and $\gs=\gs^{-1}$ is abelian, by point (i). Similarly $\gs^1\neq(0)$ otherwise $\gs^p=(0)$ for any $p\geq 1$ and $\gs^{-1}$ is a nontrivial ideal of $\gs=\gs^{-1}\oplus\gs^0$.
\vskip0.05cm\par
(iii) Let $\gm$ be a not necessarily $\bZ_2$-graded $\gs^0$-submodule of $\gs^{-1}$ and
$$
<\gm>=\gm\oplus\operatorname{span}\left\{[\gs^{p_1},[\gs^{p_2},[\ldots,[\gs^{p_q},\gm]\ldots]]]\!\!\!\!\!\!\!\!\!\!\phantom{{L^L}^L}|\;q, p_i>0\right\}
$$
the ideal generated by $\gm$ in $\gs$; it is $\bZ$-graded with $<\gm>^{-1}=\gm$. The fact that a simple Lie superalgebra does not admit any nontrivial ideal whether or not $\bZ_2$-graded (see e.g. \cite{Sc}) forces $\gm=\gs^{-1}$ or $\gm=(0)$.
\vskip0.05cm\par
(iv) it is clear from Proposition \ref{utile}.
\vskip0.05cm\par
(v) a direct consequence of (iii).
\end{proof}
By Proposition \ref{utile}, Lemma \ref{lemmetto} and the fact that $\gs^{-1}\subset\oder^{-1}(\gs)$ is stable under the adjoint action of $\oder^0(\gs)$ one gets the following.
 \begin{proposition}\label{subalgebraF}
Let $\gs=\bigoplus\gs^p$ be a grading of depth one of a simple Lie superalgebra. Any graded subalgebra $\gs\subset\gg\subset\oder(\gs)$ is of the form $\gg=\gs\niplus F$ for some graded subalgebra $F=\gg\cap\out(\gs)$ of $\out(\gs)$ and it is a transitive nonlinear  $\bZ$-graded Lie superalgebra. Moreover it is of depth one and irreducible if and only if the depth $d(F)\leq 0$ (i.e. $F$ is graded in nonnegative degrees).
\end{proposition}
The classification of depth one gradings of the simple Lie superalgebras $\gs$
was developed in \cite{K1, S1} (see also e.g. \cite{Pol}). We now recall it and, in turn, also describe the associated
gradings of $\out(\gs)$ and $\oder(\gs)$.
\vskip0.3cm\par\noindent
\subsubsection{Basic superalgebras}\hfill\vskip0.3cm\par\noindent\label{basicgrad}
Basic Lie superalgebras admit a convenient description in terms of root systems, Cartan matrices and Dynking diagrams; we recall here only the facts that we need and refer for more details to \cite{K0, K0bis, Serga, Leites}.

Let $\gs$ be a basic superalgebra, $\gt_\gs=\gh\oplus\gt_o$ a maximal torus of $\gs$ ($\gt_o\subset\out_{\bar 0}(\gs)$ is trivial with the exception of $\mathfrak{psl}(n+1|n+1)$) and $\Delta=\Delta(\gs,\gt_\gs)$ the associated root system. Then $\gs_{\bar0}$ and $\gs_{\bar1}$ decompose into the direct sum of root spaces $\gs^\alpha$ and a root $\alpha$ is called \emph{even} (resp. \emph{odd}) if $\gs_{\bar0}^\alpha$ (resp. $\gs_{\bar1}^\alpha$) is nonzero. Every root is either even or odd and the root spaces are all one-dimensional (for $\gs=\mathfrak{psl}(2|2)$ this follows from the fact that $\gt_o$ is non-trivial). 
Many properties of root systems of Lie algebras remain true for basic Lie superalgebras, see \cite[Proposition 5.3]{K0bis}. In particular, any decomposition $\Delta=\Delta^+\cup-\Delta^+$ into positive and negative roots determines a system $$\Sigma=\{\alpha_1,\ldots,\alpha_r\}$$ of simple positive roots and every positive root $\alpha\in\Delta^+$ can be written as a sum 
$$\alpha=\sum_{i=1}^{r}b_i\alpha_i$$ with non-negative integer coefficients $b_i$.

The Weyl group of the reductive Lie algebra $\gs_{\bar0}$ acts on the set of simple root systems. In contrast with the Lie algebra case this action is not transitive, and hence different simple root systems of the same 
basic Lie superalgebra may not be conjugated. 
To each orbit of the Weyl group one can associate a Dynkin diagram as follows (they were first introduced in \cite{K0, Serga}; we will use the slightly different conventions given by \cite{Leites}). 
\vskip0.2cm\par
Consider a Cartan matrix $(a_{ij})$ of order $r$ associated to $\Sigma$, see \cite{K0, Leites}. Each simple root $\alpha_i$ corresponds to a node which is colored \emph{white} $\begin{tikzpicture}
\node[root]   (1) [] {};
\end{tikzpicture}$ if $\alpha$ is even (in this case $a_{ii}=2)$, \emph{gray} $\begin{tikzpicture}
\node[xroot]   (1) [] {};
\end{tikzpicture}$ if $\alpha$ is odd and $B$-isotropic (in this case $a_{ii}=0)$, or \emph{black} 
$\begin{tikzpicture}
\node[broot]   (1) [] {};
\end{tikzpicture}$
if $\alpha$ is odd and non-isotropic (in this case $a_{ii}=1$).

The $i$-th and $j$-th nodes of the diagram are not joined if $a_{ij}=a_{ji}=0$, otherwise they are joined by 
$\operatorname{max}(|a_{ij}|,|a_{ji}|)$-edges with an arrow pointing from $\alpha_i$ to $\alpha_j$ if $|a_{ij}|<|a_{ji}|$. 
If $\gg=\mathfrak{osp}(4|2;\a)$ the Cartan matrix is integer if and only if $\a$ is an integer; in this case we illustrate just with $\a=2$.




Finally, we mark the $i$-th node with the corresponding Dynkin mark, that is the coefficient $m_{i}$ of the expression of the highest root as sum of simple roots
$$\alpha_{\text{max}}=\sum_{i=1}^{r}m_{i}\alpha_i\;. $$ 

The list of all possible Dynkin diagrams associated to basic Lie superalgebras is contained in \cite[Tables 1-5]{Leites}. 

Now, setting $\deg\alpha_i=\l_i$ for some nonnegative integer $\l_i$ and extending the definition to all roots by
\[\deg (\sum_{i=1}^{r} b_i\alpha_i)=\sum_{i=1}^{r} b_i\deg(\alpha_i)\;, \]
one gets the grading of $\gs$:
\beq
\label{allgradings}
\gs^0=\gh\oplus\bigoplus_{\substack{\alpha\in\Delta\\\deg\alpha=0}}\gs^\alpha\;,\qquad
\gs^p=\bigoplus_{\substack{\alpha\in\Delta\\\deg\alpha=p}}\gs^\alpha\;\;\;\;\;\text{for}\;\;\;\;\;p\neq 0\;.
\eeq
By \cite[Theorem 1 \& Remark 3]{K1}, all possible gradings of $\gs$ are equivalent to one of this form, for some choice of $\Sigma$. 
The depth $d(\gs)$ of $\gs$ is equal to the degree of the highest root 
$$\deg(\alpha_{\text{max}})=\sum_{i=1}^{r}b_{i,\text{max}}\deg(\alpha_i)\,$$
and coincides also with the height $\ell$ of $\gs$. Note that the Lie superalgebra \eqref{allgradings} has depth one if and only if all $\l_j=0$ with the exception of a simple root $\alpha_i$, called the {\it crossed root}, which satisfies
$m_{i}=1$ and $\l_i=1$.

The subalgebra $\gs^0$ is the direct sum of a center and a Lie superalgebra whose Cartan matrix
is obtained from the Cartan matrix of $\gs$ by removing all 
rows and columns relative to the crossed root.

Table \ref{importanttab} displays all depth one gradings of the
the basic Lie superalgebras (see \cite{S1, K1, Pol}). The symbol
\begin{tikzpicture}
\node[]  (1) 
[]            {\pie{180}};
\end{tikzpicture} 
indicates a node which can be either white or gray and $V_{\mu}$ stands for  the irreducible
module for (the semisimple part of) the Lie superalgebra $\gs^0$ with the highest weight $\mu$  and even highest vector.

We remark that some gradings in Table \ref{importanttab} are listed more than once.
We also note that the gradings of $\mathfrak{ab}(3)$ and $\mathfrak{osp}(4|2;\a)$
admit additional isomorphic presentations in terms of the following Dynkin diagrams of $\mathfrak{ab}(3)$ 
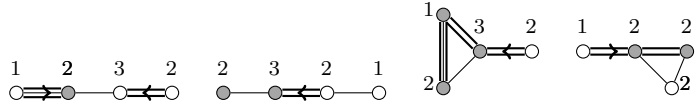
\begin{equation*}
\begin{tikzpicture}
\node[root]   (1)        [label=${{}^1}$]             {};
\node[xroot] (2) [right=of 1] {} [label=${{}^2}$]  edge [rtriplearrow] (1) edge [-] (1);
\node[root]   (3) [right=of 2] {} [label=${{}^3}$] edge [-] (2);
\node[root]   (4) [right=of 3] {} [label=${{}^2}$] edge [doublearrow] (3);
\end{tikzpicture}\;\quad
\begin{tikzpicture}
\node[xroot]   (1)        [label=${{}^2}$]             {};
\node[xroot] (2) [right=of 1] {} [label=${{}^3}$]  edge[-] (1);
\node[root]   (3) [right=of 2] {} [label=${{}^2}$] edge [doublearrow] (2);
\node[root]   (4) [right=of 3] {} [label=${{}^1}$] edge [-] (3);
\end{tikzpicture}\;\quad
\begin{tikzpicture}
\node[xroot]   (1) [label=${{}^3}$] {};
\node[root]   (2) [right=of 1][label=${{}^2}$] {} edge [doublearrow](1);
\node[xroot]   (3) [above left=of 1] [label=left:${{}^1}$] {} edge [doublenoarrow] (1);
\node[xroot]   (4) [below left=of 1] [label=left:${{}^2}$] {} edge [-] (1) edge [doublenoarrow] (3) edge [-] (3);
\end{tikzpicture}\;\quad
\begin{tikzpicture}
\node[root]   (1)        [label=${{}^1}$]             {};
\node[xroot] (2) [right=of 1] {} [label=${{}^2}$]  edge [rdoublearrow] (1);
\node[xroot]   (3) [right=of 2] {} [label=${{}^2}$] edge [doublenoarrow] (2);
\node[root]   (4) [below right=of 2] {} [label=right:${{}^2}$] edge [-] (2) edge [-] (3);
\end{tikzpicture}
\end{equation*}
and of $\mathfrak{osp}(4|2;\a)$
\begin{equation*}
\begin{tikzpicture}
\node[xroot]   (4) [right=of 3][label=${{}^1}$] {} ;
\node[xroot]   (5) [above right=of 4] [label=right:${{}^1}$] {} edge [-] (4);
\node[xroot]   (6) [below right=of 4] [label=right:${{}^1}$] {} edge [doublenoarrow] (4) edge [doublenoarrow] (5) edge [-] (5);
\end{tikzpicture}\;\qquad
\begin{tikzpicture}
\node[root](1)        [label=${{}^1}$]             {};
\node[xroot](2) [right=of 1] {} [label=${{}^2}$] edge[rdoublearrow] (1);
\node[root](3) [right=of 2]{} [label=${{}^1}$] edge[doublearrow] (2) edge [-] (2);
\end{tikzpicture}\;.
\end{equation*}
These presentations are not displayed in Table \ref{importanttab}. 
\medskip\par
\begin{sidewaystable}
\centering
\vskip13cm
{\scriptsize \begin{tabular}{|c|c|c|c|c|c|c|}
\hline
$\gs$ & Dynkin diagram & \text{order}\;$(a_{ij})$ & $\begin{gathered}\!\!\!\!\!\!\!\!\phantom{L^{1}}\text{parity of the}\\\text{number of}\;\;\begin{tikzpicture}\node[xroot](1)[label=${{}}$]{};
\end{tikzpicture}\end{gathered} $ & $\begin{gathered}\text{crossed}\\ \text{root}\end{gathered}$ & $\gs^{-1}=(\gs^1)^*$ & $\gs^{0}$ \\
\hline
\hline
&&&&\\[-3mm]
$\begin{gathered}\mathfrak{sl}(m+1|n+1)\\m< n\end{gathered}$&
$\begin{gathered}\\ \\
\begin{tikzpicture}
\node[]  (1) 
[label=${{}^1}$]            {\pie{180}};
\node[]   (2) [right=of 1] {$\;\cdots\,$} edge [-] (1);
\node[]  (3) [right=of 2] [label=${{}^1}$]{\pie{180}} edge [-] (2);
\end{tikzpicture}\end{gathered}$
&$\begin{gathered}\\ \\ m+n+1 \end{gathered}$ & $\begin{gathered}\\ \\ \text{any} \end{gathered}$ & $\begin{gathered}\\ \\ \text{any} \end{gathered}$ & 
$\begin{gathered}\\ \\  \bC^{m+1-p|n+1-q}\otimes (\bC^{p|q})^* \end{gathered}$ & $\mathfrak{s}(\ggl(m+1-p|n+1-q)\oplus\ggl(p|q))$ 
\\
\cline{1-1}\cline{7-7}
&&&&\\[-2mm]
$\begin{gathered}\mathfrak{psl}(m+1|n+1)\\ m=n\neq 0\end{gathered}$&
&  &  &  &  & $\mathfrak{ps}(\ggl(m+1-p|n+1-q)\oplus\ggl(p|q))$ 
\\
\hline
&&&&\\[-2mm]
&
\begin{tikzpicture}
\node[]   (3) [label=${{}^1}$] {\pie{180}};
\node[]   (4) [right=of 3] {\pie{180}} [label=${{}^2}$]  edge [-] (3);
\node[]   (5) [right=of 4] {$\;\cdots\,$} edge [-] (4);
\node[]   (6) [right=of 5] {\pie{180}} [label=${{}^2}$] edge [-] (5);
\node[root]   (7) [right=of 6] {} [label=${{}^2}$]  edge [rdoublearrow] (6);
\end{tikzpicture}& $\begin{gathered} m+n \\ (m+n\geq 3) \end{gathered}$ & $0$ & \text{first} &
$\bC^{2m-1|2n}$ & $\mathfrak{cosp}(2m-1|2n)$ 
\\
\cline{2-7}

&&&&\\[-3mm]
$\begin{gathered}\mathfrak{osp}(2m+1|2n)\\ m,n\geq1\\ \\ \end{gathered}$
&
\begin{tikzpicture}
\node[]   (3) [label=${{}^1}$] {\pie{180}};
\node[]   (4) [right=of 3] {\pie{180}} [label=${{}^2}$]  edge [-] (3);
\node[]   (5) [right=of 4] {$\;\cdots\,$} edge [-] (4);
\node[]   (6) [right=of 5] {\pie{180}} [label=${{}^2}$] edge [-] (5);
\node[broot]   (7) [right=of 6] {} [label=${{}^2}$]  edge [-] (6);
\end{tikzpicture}& m+n & 1 & \text{first}&
$\bC^{2m-1|2n}$ & $\mathfrak{cosp}(2m-1|2n)$ 
\\
\hline

&&&&\\[-5mm]
&$\begin{gathered}\\
\begin{tikzpicture}
\node[]   (1) [label=${{}^1}$] {\pie{180}};
\node[]   (2) [right=of 1][label=${{}^2}$] {\pie{180}} edge [-] (1);
\node[]   (3) [right=of 2] {$\;\cdots\,$} edge [-] (2);
\node[]   (4) [right=of 3][label=${{}^2}$] {\pie{180}} edge (3);
\node[xroot]   (5) [above right=of 4] [label=right:${{}^1}$] {} edge [-] (4);
\node[xroot]   (6) [below right=of 4] [label=right:${{}^1}$] {} edge [-] (4) edge [doublenoarrow] (5);
\end{tikzpicture}\end{gathered}$&$\begin{gathered}\\ \\ m+n \end{gathered}$ &$\begin{gathered}\\ \\ 1\end{gathered}$&$\begin{gathered} \\ \\ \text{first}\\ \\ \text{last}\end{gathered}$&
$\begin{gathered}\\ \\ \bC^{2m-2|2n}\\ \\ \L^2(\bC^{m|n})\end{gathered}$ & $\begin{gathered} \\ \\ \mathfrak{cosp}(2m-2|2n)\\ \\ \mathfrak{gl}(m|n)\end{gathered}$ 
\\
\cline{2-7}

&&&&\\[-2mm]
&
\begin{tikzpicture}
\node[]   (1) [label=${{}^2}$] {\pie{180}};
\node[]   (2) [right=of 1][label=${{}^2}$] {\pie{180}} edge [-] (1);
\node[]   (3) [right=of 2] {$\;\cdots\,$} edge [-] (2);
\node[]   (4) [right=of 3][label=${{}^2}$] {\pie{180}} edge (3);
\node[xroot]   (5) [above right=of 4] [label=right:${{}^1}$] {} edge [-] (4);
\node[xroot]   (6) [below right=of 4] [label=right:${{}^1}$] {} edge [-] (4) edge [doublenoarrow] (5);
\end{tikzpicture}& $\begin{gathered} m+n \\ \\ \end{gathered}$&$\begin{gathered} 0\\ \\ \end{gathered}$&$\begin{gathered} \text{last}\\ \\ \end{gathered}$&$\begin{gathered} \L^2(\bC^{m|n})\\ \\ \end{gathered}$
 &$\begin{gathered} \ggl(m|n)\\  \\ \end{gathered}$
\\
\cline{2-7}

&&&&\\[0mm]
$\begin{gathered}\mathfrak{osp}(2m,2n) \\ m,n\geq1\\ m+n\geq 3\end{gathered}$
&
\begin{tikzpicture}
\node[]   (1) [label=${{}^1}$] {\pie{180}};
\node[]   (2) [right=of 1][label=${{}^2}$] {\pie{180}} edge [-] (1);
\node[]   (3) [right=of 2] {$\;\cdots\,$} edge [-] (2);
\node[]   (4) [right=of 3][label=${{}^2}$] {\pie{180}} edge (3);
\node[root]   (5) [above right=of 4] [label=right:${{}^1}$] {} edge [-] (4);
\node[root]   (6) [below right=of 4] [label=right:${{}^1}$] {} edge [-] (4);
\end{tikzpicture}&m+n & 0 & $\begin{gathered} \text{first}\\ \\ \text{last}\\ \end{gathered}$&
$\begin{gathered} \bC^{2m-2|2n}\\ \\ \L^2(\bC^{m|n})\\ \end{gathered}$ & $\begin{gathered} \mathfrak{cosp}(2m-2|2n)\\ \\ \mathfrak{gl}(m|n)\\ \end{gathered}$ 
\\
\cline{2-7}

&&&&\\[-1mm]
&
\begin{tikzpicture}
\node[]   (1) [label=${{}^2}$] {\pie{180}};
\node[]   (2) [right=of 1][label=${{}^2}$] {\pie{180}} edge [-] (1);
\node[]   (3) [right=of 2] {$\;\cdots\,$} edge [-] (2);
\node[]   (4) [right=of 3][label=${{}^2}$] {\pie{180}} edge (3);
\node[root]   (5) [above right=of 4] [label=right:${{}^1}$] {} edge [-] (4);
\node[root]   (6) [below right=of 4] [label=right:${{}^1}$] {} edge [-] (4);
\end{tikzpicture}& $\begin{gathered} m+n \\ \\ \end{gathered}$ & $\begin{gathered} 1\\ \\ \end{gathered}$& $\begin{gathered} \text{last}\\ \\ \end{gathered}$ & $\begin{gathered} \L^2(\bC^{m|n})\\ \\ \end{gathered}$
 & $\begin{gathered} \ggl(m|n)\\ \\ \end{gathered}$
\\
\cline{2-7}

&&&&\\[-2mm]
&
\begin{tikzpicture}
\node[]   (1)    [label=${{}^1}$]                 {\pie{180}};
\node[] (2) [right=of 1] [label=${{}^2}$]{\pie{180}} edge [-] (1);
\node[]   (3) [right=of 2] {$\;\cdots\,$} edge [-] (2);
\node[]   (4) [right=of 3][label=${{}^2}$]{\pie{180}}  edge [-] (3);
\node[root]   (5) [right=of 4] [label=${{}^1}$]{} edge [doublearrow] (4);
\end{tikzpicture}&m+n & 1& 
$\begin{gathered} \text{first}\\ \\ \text{last}\end{gathered}$&
$\begin{gathered} \bC^{2m-2|2n}\\ \\ \L^2(\bC^{m|n})\end{gathered}$ & $\begin{gathered} \mathfrak{cosp}(2m-2|2n)\\ \\ \mathfrak{gl}(m|n)\end{gathered}$ 
\\
\cline{2-7}

&&&&\\[-2mm]
&
\begin{tikzpicture}
\node[]   (1)    [label=${{}^2}$]                 {\pie{180}};
\node[] (2) [right=of 1] [label=${{}^2}$]{\pie{180}} edge [-] (1);
\node[]   (3) [right=of 2] {$\;\cdots\,$} edge [-] (2);
\node[]   (4) [right=of 3][label=${{}^2}$]{\pie{180}}  edge [-] (3);
\node[root]   (5) [right=of 4] [label=${{}^1}$]{} edge [doublearrow] (4);
\end{tikzpicture}&m+n & 0&\text{last} & $\L^2(\bC^{m|n})$
 & $\ggl(m|n)$ 
\\
\hline

&&&&\\[-2mm]
$\mathfrak{ab}(3)$&
\begin{tikzpicture}
\node[xroot]   (1)        [label=${{}^2}$]             {};
\node[root] (2) [right=of 1] {} [label=${{}^3}$]  edge[-] (1);
\node[root]   (3) [right=of 2] {} [label=${{}^2}$] edge [doublearrow] (2);
\node[root]   (4) [right=of 3] {} [label=${{}^1}$] edge [-] (3);
\end{tikzpicture} & $4$ & $1$ & \text{last} &
$V_{-\e_1+\d_1+\d_2}$ & $\mathfrak{cosp}(2|4)$ \\
\hline

&&&&\\[-3mm]
$\begin{gathered}\\ \mathfrak{osp}(4|2;\a)\\ \a\neq0,\pm1,-2,-1/2\end{gathered}$
&
\begin{tikzpicture}
\node[root](1)        [label=${{}^1}$]             {};
\node[xroot](2) [right=of 1] {} [label=${{}^2}$] edge[-] (1);
\node[root](3) [right=of 2]{} [label=${{}^1}$] edge[doublearrow] (2);
\end{tikzpicture}& $3$ & $1$ & $\begin{gathered} \text{first} \\ \text{last} \end{gathered}$ &
$\begin{gathered} V_{-\a^{-1}\e_1}\\V_{-\a\e_1}\;\,\,\,\,\,\,\end{gathered}$ & $\ggl(1|2)$ \\
\cline{2-7}

&&&&\\[-2mm]
&
\begin{tikzpicture}
\node[root](1)        [label=${{}^1}$]             {};
\node[xroot](2) [right=of 1] {} [label=${{}^2}$] edge[rdoublearrow] (1) edge[-] (1);
\node[root](3) [right=of 2]{} [label=${{}^1}$] edge[-] (2);
\end{tikzpicture}& $3$ & $1$ & \text{last} &
$\begin{gathered} V_{(1+\a)^{-1}\e_1} \end{gathered}$ & $\ggl(1|2)$ \\
\hline 

\end{tabular}}
\vskip0.7cm
\caption[]{\label{importanttab} The depth one gradings of the basic Lie superalgebras.}
\end{sidewaystable}
\clearpage
\par\noindent
Now $\gs=\oder(\gs)$ for all basic Lie superalgebras, except $\gs=\mathfrak{psl}(n+1|n+1)$. 

If $n\geq 2$ then $\oder(\gs)\simeq\mathfrak{pgl}(n+1|n+1)$. More precisely $\out(\gs)=\gt_o$ is generated by an even element $Z$ which acts trivially on $\gs_{\bar 0}$ and with eigenvalues $\pm 1$ on the two components of $\gs_{\bar 1}$ (recall Table \ref{Kactab}). In particular $Z$ has degree zero for all gradings of $\gs$ and $\oder^0(\gs)\simeq \mathfrak{p}(\ggl(n+1-p|n+1-q)\oplus\ggl(p|q))$. 

If $\gs=\mathfrak{psl}(2|2)$ then $\oder(\gs)\simeq\gs\niplus\mathfrak{sl}(2)$ where
$\mathfrak{sl}(2)$ is generated by an element $Z$ as above and the two nilpotent even derivations $Z_{\pm}$ 
given by
\begin{align*}
Z_{\pm}&(\gs_{\bar 0})=Z_+(\bC^2\otimes (\bC^2)^*)=Z_-((\bC^2)^*\otimes\bC^2)=(0)\;,\\
Z_+&:(\bC^2)^*\otimes\bC^2\longrightarrow\bC^2\otimes (\bC^2)^*\quad\text{is an isomorphism of}\;\gs_{\bar 0}-\text{modules}\;,\\
Z_-&\big|_{\bC^2\otimes (\bC^2)^*} =(Z_+\big|_{(\bC^2)^*\otimes\bC^2})^{-1}:\bC^2\otimes (\bC^2)^*\longrightarrow(\bC^2)^*\otimes\bC^2\;.
\end{align*}
A direct computation which uses an explicit decomposition of $\mathfrak{psl}(2|2)$ in root spaces gives Table \ref{TABSL2} with the gradings of $\mathfrak{psl}(2|2)$ and $\sl(2)=\left\langle Z_+,Z,Z_-\right\rangle$.
{\small
\begin{table}[H]
\begin{centering}
\makebox[\textwidth]{%
\begin{tabular}{|c|c|c|c|c|c|}
\hline
$\gs$ & Dynkin diagram & Grading & $\deg(Z)$ & $\deg(Z_+)=-\deg(Z_-)$\\
\hline
\hline
&&&&\\[-3mm]
$\begin{gathered}\mathfrak{psl}(2|2)\end{gathered}$&
$\begin{gathered}
\begin{tikzpicture}
\node[root]   (1)       [label=${{}^1}$]              {};
\node[xroot]   (2) [right=of 1] [label=${{}^1}$] {} edge [-] (1);
\node[root]   (3) [right=of 2] [label=${{}^1}$] {} edge [-] (2);
\end{tikzpicture}\\

\begin{tikzpicture}
\node[xroot]   (1)       [label=${{}^1}$]              {};
\node[root]   (2) [right=of 1] [label=${{}^1}$] {} edge [-] (1);
\node[xroot]   (3) [right=of 2] [label=${{}^1}$] {} edge [-] (2);
\end{tikzpicture}\\
\begin{tikzpicture}
\node[xroot]   (1)       [label=${{}^1}$]              {};
\node[xroot]   (2) [right=of 1] [label=${{}^1}$] {} edge [-] (1);
\node[xroot]   (3) [right=of 2] [label=${{}^1}$] {} edge [-] (2);
\end{tikzpicture}\end{gathered}$
& $\deg(\a_i)=\l_i\geq 0$ & $0$ & $\begin{gathered} \l_1+2\l_2+\l_3\\ \l_1-\l_3^{\phantom{C}}\\ \l_1+\l_3^{\phantom{C}}\end{gathered}$ 
\\
\hline
\end{tabular}}
\end{centering}
\caption[]{\label{TABSL2} The gradings of $\oder(\gs)=\mathfrak{psl}(2|2)\niplus\sl(2)$.} \vskip14pt
\end{table}}
\vskip-0.5cm\par\noindent
\begin{proposition}
All transitive nonlinear irreducible graded Lie subalgebras $\gs\subset\gg\subset\oder(\gs)$ of depth $1$, 
$\gs=\mathfrak{psl}(2|2)$, are of the form $\gg=\gs\niplus F$ where $\gs=\bigoplus\gs^p$ is in Table \ref{TABSL22} and $F$ is a nonnegatively graded subalgebra of $\sl(2)$.
\end{proposition}
{\scriptsize
\begin{table}[H]
\begin{centering}
\makebox[\textwidth]{%
\begin{tabular}{|c|c|c|c|c|c|c|c|}
\hline
Dynkin diagram & crossed root & $\gs^{-1}=(\gs^1)^*$ & $\gs^0$ & $\sl(2)^0$ & $\sl(2)^1$ & $\sl(2)^2$\\
\hline
\hline
&&&&\\[-3mm]
$\begin{gathered}
\begin{tikzpicture}
\node[root]   (1)       [label=${{}^1}$]              {};
\node[xroot]   (2) [right=of 1] [label=${{}^1}$] {} edge [-] (1);
\node[root]   (3) [right=of 2] [label=${{}^1}$] {} edge [-] (2);
\end{tikzpicture}\\

\begin{tikzpicture}
\node[root]   (1)       [label=${{}^1}$]              {};
\node[xroot]   (2) [right=of 1] [label=${{}^1}$] {} edge [-] (1);
\node[root]   (3) [right=of 2] [label=${{}^1}$] {} edge [-] (2);
\end{tikzpicture}\\

\begin{tikzpicture}
\node[xroot]   (1)       [label=${{}^1}$]              {};
\node[xroot]   (2) [right=of 1] [label=${{}^1}$] {} edge [-] (1);
\node[xroot]   (3) [right=of 2] [label=${{}^1}$] {} edge [-] (2);
\end{tikzpicture}
\end{gathered}$
& $\begin{gathered} \text{first}_{\phantom{C}}\\ \text{second}^{\phantom{C}^{\phantom{C}}}\\ \text{second}^{\phantom{C}^{\phantom{C}}}  \end{gathered}$ & $\begin{gathered} \bC^{1|2} \\ \bC^{2|0}\otimes(\bC^{0|2})^* \\ \bC^{1|1}\otimes(\bC^{1|1})^*  \end{gathered}$ & $\begin{gathered} \mathfrak{sl}(1|2) \\ \sl(2)\oplus \sl(2) \\ \sl(1|1)\oplus\sl(1|1)\end{gathered}$ & $\begin{gathered} Z\\ Z\\ \sl(2) \end{gathered}$ & $\begin{gathered}Z_+\\ (0)\\ (0)\end{gathered}$ & $\begin{gathered}(0) \\ Z_+ \\ (0)\end{gathered}$ 
\\
\hline
\end{tabular}}
\end{centering}
\caption[]{\label{TABSL22}} \vskip14pt 
\end{table}}
\vskip-0.5cm\par\noindent
\begin{proof}
At once from Proposition \ref{subalgebraF} and Table \ref{TABSL2}.
\end{proof}

\vskip0.3cm\par
\subsubsection{Strange superalgebras}\hfill\vskip0.3cm\par\noindent\label{strangegrad}
The gradings of $\gs=\mathfrak{psq}(n)$ are in a one-to-one correspondence with those of $\sl(n)\simeq\gs_{\bar 0}\simeq\gs_{\bar 1}$ and are of depth one if one and only one root of the Dynkin diagram of $\mathfrak{sl}(n)$ is crossed. 
{\scriptsize
\begin{table}[H]
\begin{centering}
\makebox[\textwidth]{%
\begin{tabular}{|c|c|c|c|c|}
\hline
$\gs$ & Dynkin diagram of $\mathfrak{sl}(n)$  & $\begin{gathered}\text{crossed}\\ \text{root}\end{gathered}$ & $\gs^{-1}=(\gs^1)^*$ & $\gs^{0}$ \\
\hline
\hline
&&&&\\[-3mm]
$\begin{gathered}\mathfrak{psq}(n)\\ n\geq 3\end{gathered}$&
\begin{tikzpicture}
\node[root]  (1) 
[label=${{}^1}$]            {};
\node[]   (2) [right=of 1] {$\;\cdots\,$} edge [-] (1);
\node[root]  (3) [right=of 2] [label=${{}^1}$]{} edge [-] (2);
\end{tikzpicture}
& \text{p-{th}} & $\bC^{p|p}\circledcirc(\bC^{n-p|n-p})^*$ & $\mathfrak{ps}(\mathfrak{q}(p)\oplus\mathfrak{q}(n-p))$
\\
\hline
\end{tabular}}
\end{centering}
\caption[]{\label{primatab34} The depth one gradings of $\mathfrak{psq}(n)$.}
\end{table}}
\par\noindent
In this case $\oder(\gs)=\mathfrak{pq}(n)$ and $\out(\gs)$ is generated by an odd derivation $D$ satisfying
$$
D(\gs_{\bar 0})=(0)\;,\qquad D\big|_{\gs_{\bar 1}}:\gs_{\bar 1}\mapsto \gs_{\bar 0}\;\;\;\text{is an isomorphism of}\;\;\gs_{\bar 0}-\text{modules}\,.
$$
It has degree zero for all gradings of $\gs$ and $\oder^0(\gs)\simeq \mathfrak{p}(\mathfrak{q}(p)\oplus\mathfrak{q}(n-p))$.
\bigskip\par
The gradings of $\gs=\mathfrak{spe}(n)$ are in a one-to-one correspondence with the gradings of $ \mathfrak{sl}(n)\simeq\mathfrak{spe}(n)_{\bar 0}$
and an integer $k$ which determines the degree $\deg(F)=k$ of the highest weight vector $F$ of the $\mathfrak{sl}(n)$-module $S^2(\bC^n)$. 

Table \ref{primatab35} below displays all depth one gradings of $\mathfrak{spe}(n)$; therein {\tiny $\begin{pmatrix} n& 0\\ 0 &n-2 \end{pmatrix}$} is the diagonal matrix of order $2(n-1)$ with the eigenvalues $n$ and $n-2$, each of the same multiplicity $n-1$. 

{\scriptsize
\begin{table}[H]
\begin{centering}
\makebox[\textwidth]{%
\begin{tabular}{|c|c|c|c|c|c|c|c|}
\hline
$\gs$ & Dynkin diagram of $\mathfrak{sl}(n)$  & $\begin{gathered}\text{crossed}\\ \text{root}\end{gathered}$ & $k$ & $\gs^{-1}$ & $\gs^{0}$ & $\gs^{1} $ & $\gs^{2}$ \\
\hline
\hline

&&&&\\[-3mm]
$$&
& \text{none} & $-1$ & $\Pi(S^2(\bC^n))$ & $\sl(n)$ & $\Pi(\L^2((\bC^n)^*))$ & $(0)$
\\
\cline{3-8}

&&&&\\[-3mm]
$\begin{gathered}\\ \\ \mathfrak{spe}(n)\\ n\geq 3\end{gathered}$&
$\begin{gathered}\\ \\
\begin{tikzpicture}
\node[root]  (1) 
[label=${{}^1}$]            {};
\node[]   (2) [right=of 1] {$\;\cdots\,$} edge [-] (1);
\node[root]  (3) [right=of 2] [label=${{}^1}$]{} edge [-] (2);
\end{tikzpicture}
\end{gathered}$
& \text{p-th} & $1$ & $\Pi(S^2(\bC^{n-p|p}))$ & $\sl(n-p|p)$ & $\Pi(\L^2((\bC^{n-p|p})^*))$ & $(0)$
\\
\cline{3-8}

&&&&\\[-3mm]
$$&

& \text{none} & $1$ & $\Pi(S^2(\bC^{0|n}))$ & $\sl(n)$ & $\Pi(\L^2((\bC^{0|n})^*))$ & $(0)$
\\
\cline{3-8}

&&&&\\[-3mm]
$$&

& \text{first} & $2$ & $\bC^{n-1|n-1}$ & $\mathfrak{spe}(n-1)\niplus\begin{pmatrix} n& 0\\ 0 &n-2 \end{pmatrix}$ & $(\bC^{n-1|n-1})^*$ & $\bC^{0|1}$
\\
\hline
\end{tabular}}
\end{centering}
\caption[]{\label{primatab35} The depth one gradings of $\mathfrak{spe}(n)$.}
\end{table}}
\par\noindent
In this case $\oder(\gs)\simeq\mathfrak{pe}(n)$ where $\out(\gs)=\gt_o$ is generated by 
the even derivation
$$
D(\gs_{\bar 0})=(0)\;,\qquad D\big|_{S^2(\bC^n)}=\Id\;,\qquad
D\big|_{\L^2((\bC^n)^*)}=-\Id\;.
$$
It has always degree zero and $\oder^0(\gs)\simeq \ggl(n-p|p)$ ($0\leq p\leq n$) or $\mathfrak{cpe}(n-1)$.

\vskip0.3cm\par
\subsubsection{Cartan superalgebras}\hfill\vskip0.3cm\par\noindent\label{cartangrad}
The gradings of $W(n)$, $S(n)$, $\widetilde{S}(n)$ and $H(n)$ are all generated by gradings of type $\vec k=(k_1,\ldots,k_m)$ which satisfy appropriate restrictions on the spectrum (see \cite{K1} and also \cite[\S 4.1]{AS}). The list of those of depth one is contained in e.g. \cite[Proposition 9.1]{CKIMRN} and displayed in Table \ref{ultimatavoletta}, with the action of $\gs^0$ on $\gs^{-1}$.
\begin{sidewaystable}
\centering
\vskip13cm
{\scriptsize \begin{tabular}{|c|c|c|c|c|c|}
\hline
$\gs^{\phantom{L}}$ & $\dim U^{-1^{\phantom{L}}}$ & $\dim U^{0^{\phantom{L}}}$ & $\dim U^{1^{\phantom{L}}}$ & $\gs^{-1^{\phantom{L}}}$ & $\gs^{0^{\phantom{L}}}$ \\
\hline
\hline
&&&&\\[-3mm]
$\begin{gathered}\\ \\ W(n)\\n\geq 3\end{gathered}$&
$1$
& $n-1$ & $0$ & $\Pi(W(n-1))$ & $W(n-1)\inplus \L(n-1)$ 
\\
\cline{2-6}

&&&&\\[-0mm]
&
$0$ & $\begin{gathered} r \\ (0\leq r\leq n-1) \end{gathered}$ & $n-r$ & $\L(r)\otimes \bC^{0|n-r}$ &
$W(r)\inplus \L(r)\otimes \ggl(n-r)$  
\\
\hline

&&&&\\[-3mm]
$\begin{gathered}\\ \\ S(n)\\n\geq 4\end{gathered}$&
$1$
& $n-1$ & $0$ & $\Pi(S(n-1))$ & $S(W(n-1)\inplus \L(n-1))$ 
\\
\cline{2-6}

&&&&\\[-0mm]
&
$0$ & $\begin{gathered} r \\ (0\leq r\leq n-1) \end{gathered}$ & $n-r$ & $\L(r)\otimes \bC^{0|n-r}$ &
$S(W(r)\inplus \L(r)\otimes \ggl(n-r))$  
\\
\hline

&&&&\\[+2mm]
$$ &
$0$
& $0$ & $n$ & $\bC^{0|n}$ & $\mathfrak{so}(n)$ 
\\
\cline{2-6}

&&&&\\[-2mm]
$\begin{gathered}\\ H(n)\\n\geq 5\end{gathered}$
&
$1$ & $n-2$ & $1$ & $\Pi(\L(n-2))$ &
$H'(n-2)\inplus \L(n-2)/\vol$  
\\
\cline{2-6}

&&&&\\[-2mm]
&
$0$ & $\begin{gathered} \frac{n}{2}\\ (n\;\text{even})\end{gathered}$ & $\begin{gathered} \frac{n}{2}\\ (n\;\text{even})\end{gathered}$ & $\L(\frac{n}{2})/\bC$ & 
$W(\frac{n}{2})$  
\\
\hline

\end{tabular}}
\vskip0.7cm
\caption[]{\label{ultimatavoletta} The depth one gradings of the Cartan Lie superalgebras and their nonpositive parts.}
\end{sidewaystable}
\clearpage
\par\noindent
Now $\gs=\oder(\gs)$ if $\gs=W(n)$. 
If $\gs=S(n)$ and $H(n)$ then $\oder(\gs)$ is isomorphic to, respectively, the superalgebra $CS(n)$ of  vector fields of {\it constant} divergence and the \emph{full} Hamiltonian superalgebra $H'(n)$ (the simple Lie superalgebra $H(n)$ is the derived ideal of $H'(n)$).

If $\gs=S(n)$ then $\out(\gs)$ consists of the ``principal'' Euler supervector field 
\beq
\label{Euler}
\sum_{\a=1}^n\xi^\a\frac{\partial}{\partial \xi^\a}\;;
\eeq
it has always degree zero and
$\oder^0(\gs)\simeq
C(\gs^0)$.

If $\gs=H(n)$ then $\out(\gs)$ is the two-dimensional solvable Lie superalgebra generated by
\eqref{Euler} and the Hamiltonian supervector field
\beq
\label{volume}
H_{f}=-(-1)^{p(f)}\;\sum_{\alpha=1}^{n}\frac{\partial f}{\partial \xi^\alpha} \frac{\partial}{\partial \xi^{n+1-\alpha}}
\eeq
with $f=\xi^1\cdots\xi^n$. This supervector field has degree $n-2$ in the first grading of $H(n)$ in Table \ref{ultimatavoletta} and
$\oder^0(\gs)=\mathfrak{co}(n)$ in this case (only the Euler field contributes to the zero-degree part). 

It has degree zero and, respectively, $n/2-1$ (n even) in the second and last gradings of $H(n)$ in Table \ref{ultimatavoletta} and contributes to $\oder^0(\gs)$ only in the second grading. A direct computation using the presentation $H'(n)\simeq\L^+$ given by \eqref{volume} and the Poisson bracket (see \cite{ChK} for its definition and basic properties) says that $\oder^0(\gs)=(H'(n-2)\inplus \L(n-2))\oplus G$ and $W(\frac{n}{2})\oplus D$ in these two cases, where $G$ is the ``principal'' grading operator acting on $\L(n-2)$ and $D$ is the grading operator of the last grading of $\gs=H(n)$ in Table \ref{ultimatavoletta}.
\bigskip\par
The following is a direct consequence of the classification of the gradings of $\oder(\gs)$ carried out in \S \ref{basicgrad}-\S\ref{cartangrad}. 
\begin{proposition}\label{aeN1}
Let $\gs$ be a simple Lie superalgebra different from $\mathfrak{psl}(2|2)$. If $\gs=\bigoplus\gs^p$ is a grading of depth one and $\oder(\gs)=\bigoplus \oder^p(\gs)$ the associated grading of $\oder(\gs)$ then $d(\oder(\gs))=1$ and $\oder^{-1}(\gs)=\gs^{-1}$. If $\gs=\mathfrak{psl}(2|2)$ then
$d(\oder(\gs))\leq 2$. 
\end{proposition}

The results of Proposition \ref{subalgebraF} and Proposition \ref{aeN1} together with Proposition \ref{torostand} applied
to the case where the socle is a simple Lie superalgebra $\gs$ immediately yield the following main result.
\begin{theorem}\label{importantcase}
Any transitive nonlinear irreducible $\bZ$-graded Lie superalgebra of depth $1$
and with a simple socle $\gs$ is of the form $\gg=\gs\niplus F$
where $\gs=\bigoplus\gs^p$ is a grading of depth $1$ of a simple Lie superalgebra
and  $F=\bigoplus F^p$ a graded subalgebra of $\out(\gs)$ of depth $d(F)\leq 0$.
Any grading of $F$ satisfies $d(F)\leq 0$ (i.e. $F$ is graded in nonnegative degrees) for all $\gs\neq\mathfrak{psl}(2|2)$. 
\end{theorem}
The gradings of depth $1$ of the simple Lie superalgebras with the associated gradings of the algebras of outer derivations are classified in \S \ref{basicgrad}-\S\ref{cartangrad}. Theorem \ref{importantcase} gives therefore a complete description of the {\it transitive nonlinear irreducible $\bZ$-graded Lie superalgebras of depth $1$
with a simple socle}.
\vskip0.3cm\par\noindent
\section{The general case}\label{sec:4}
We now describe all transitive gradings $\gg=\bigoplus\gg^p$ of depth one 
of a Lie superalgebra $\gg$ such that $\gg^1\neq (0)$ and the representation $\mathrm{ad}_{\gg^0}|_{\gg^{-1}}$ is irreducible. 

By Theorem \ref{block}, 
Proposition \ref{graduato} and Proposition \ref{torostand} $\gg$ is an intermediate admissible Lie superalgebra and the grading is induced by a grading of $\gs_{\max}$
generated by a grading of $\gs$ and a grading of $\L$ and $W$ of type $\vec k=(k_1,\ldots,k_m)$. 
\bigskip\par
We start with the following preliminary result.
\begin{proposition}\label{firsthm}
Any transitive and irreducible grading of depth $1$ of an intermediate admissible Lie superalgebra $\gg$ with socle $\gs\otimes\Lambda$ is induced by a grading of $\gs_{\max}$ generated by:
\begin{itemize}
\item[(I)] a depth one grading of $\gs$ and the trivial grading $\vec k=(0)$ of $\L$ and $W$ or
\item[(II)] the trivial grading of $\gs$ and the $\vec k=(-1,0)$ grading of $\L$ and $W$ with $\dim U^{-1}=1$.
\end{itemize}
\end{proposition}
\begin{proof}
The inclusion $\gg\supset\gs^\L= \gs\otimes\L$ immediately implies
$$
\gg^{p}\supset(\gs^\L)^{p}=\bigoplus_{k\in\bZ}\gs^{p+k}\otimes\L^{-k}\;,
$$
where $\L=\bigoplus\L^p$ and $W=\bigoplus W^p$ are the gradings of type $\vec k=(k_1,\ldots,k_m)$ and $\gs=\bigoplus\gs^p$ is the grading of $\gs$. 
Since $d(\gg)=1$ and $\L$ has always the identity element $1$ with $\deg 1=0$,
one gets two cases:
\begin{itemize}
\item[(I)] $d(\gs)=1$ and $d(\L)=0$,
\item[(II)] $d(\gs)\leq 0$.
\end{itemize}
We consider them separately. 
\vskip0.25cm\par
{\it Case (I)}. The subalgebra
$$\gs_{\max}^0=\bigoplus_{p<0}(\oder^{p}(\gs)\otimes\L^{-p})\oplus(\oder^0(\gs)\otimes\L^0)\oplus W^0$$
stabilizes the nonzero subspace $\gs^{-1}\otimes\L^0\subset \gs_{\max}^{-1}$. In particular 
this subspace is $\gg^0$-stable and hence $\gg^{-1}=\gs^{-1}\otimes\L^0$, by irreducibility of the grading of $\gg$. 
By $d(\L)=0$ and Lemma \ref{gradL0} one also has all $k_\a\geq 0$.

If $\gg=\gs_{\max}=\oder(\gs)\otimes\L\niplus\textbf{1}\otimes W$ one has $d(W)=0$ by irreducibility and $\vec k=(0)$ by Lemma \ref{gradL0}; the general case is an appropriate modification of this argument, as we will now see.

First of all assume $k_m>0$ by contradiction. We recall that
$\gg$ is admissible and that any constant supervector field $\partial_{\xi}\in\partial_V$ is related to an $x\in\gg$ such that $\pi(x)\in W=\partial_V\oplus\L^+\partial_V$ projects on $\partial_\xi$, i.e. 
\beq
\label{related}
\begin{split}
x\equiv \pi(x)\!\!\mod\!\!\oder(\gs)\otimes\Lambda\;,\\
\text{with}\;\;\pi(x)\equiv \partial_{\xi}\!\!\mod\!\Lambda^+\partial_V\;,
\end{split}
\eeq
where $\pi:\gs_{\max}\to W$ is the natural projection. Let $x\in\gg$ be related to $\partial_{\xi^m}$; one may assume $\deg(x)=-k_m$ since 
$\deg(\partial_{\xi^m})=-k_m$ and $\gg$ is $\bZ$-graded. This implies $x\in\gg^{-1}$ since $d(\gg)=1$;
however $x\notin\gs^{-1}\otimes\L^0$, a contradiction.

It follows $k_m=0$, $\vec k=(0)$ and $\gg^{-1}=\gs^{-1}\otimes\L$.
\vskip0.25cm\par
{\it Case (II)}. In this case $\gs_+=\bigoplus_{p>0}\gs^p$ is a nilpotent ideal of $\gs$, hence $\gs_+=(0)$ by simplicity of $\gs$ and both $\gs=\gs^0$ and
$\mathrm{der}(\gs)=\mathrm{der}^0(\gs)$ are trivially graded. 
If $d(\L)\leq 0$ then $\gs\otimes\L$ is an ideal of $\gg$ contained in $\gg_0=\bigoplus_{p\geq 0}\gg^p$, a possibility which is not allowed by the transitivity of the grading of $\gg$. It follows that $d(\L)=1$  and $\vec k=(-1,k_2,\ldots,k_m)$ with $\dim U^{-1}=1$, from Lemma \ref{gradL0}. 

Now, the subalgebra
$$\gs_{\max}^0=(\oder(\gs)\otimes\L^{0})\oplus W^0$$
stabilizes the nontrivial subspace $\gs\otimes\L^{-1}\subset\gs^{-1}_{\max}$ and therefore
$
\gg^{-1}=\gs\otimes\L^{-1}
$. This fact, admissibility of $\gg$ and a similar argument to the last part of (I) 
finally imply $\vec k=(-1,0)$ with $\dim U^{-1}=1$.
\end{proof}
It is convenient to have a closer look to the gradings of type  (I) and (II) in the case $\gg=\gs_{\max}$. We recall that $\out(\gs)$ is a graded subalgebra, by Proposition \ref{utile}.
\vskip0.3cm\par
{\it Case (I)}. Here $\gs_{\max}^{-p}=(0)$ for all $p\geq 3$ and
\beq
\label{case(i)max}
\begin{split}
\gs_{\max}^{-2}&=\out^{-2}(\gs)\otimes\L,\phantom{ccccccccccccccccccccccccccccccccc}\\
\gs_{\max}^{-1}&=(\gs^{-1}\otimes\L) \oplus (\out^{-1}(\gs)\otimes\L),\\
\gs_{\max}^{0}&=(\gs^{0}\otimes\L) \oplus (\out^{0}(\gs)\otimes\L)\oplus W,\\
\gs_{\max}^{1}&=(\gs^{1}\otimes\L)\oplus (\out^{1}(\gs)\otimes\L),\\
\gs_{\max}^{p}&=(\gs^{p}\otimes\L)\oplus(\out^{p}(\gs)\otimes\L)\;\;(p\geq 2),
\end{split}
\eeq
where $\out^{-2}(\gs)=\out^{-1}(\gs)=(0)$ for all $\gs$ with the exception of $\gs=\mathfrak{psl}(2|2)$, see Proposition \ref{aeN1}.
\vskip0.3cm\par
{\it Case (II)}. The grading of $\Lambda$ and $W$ has type $\vec k=(-1,0)$ and defined by the grading $U=U^{-1}\oplus U^0=\bR\xi\oplus E$. Hence $\L=\L^{-1}\oplus\L^0$, $W=W^{-1}\oplus W^0\oplus W^1$ where
\begin{multline*}
\!\!\!\L^{-1}=\xi\L(E)\;,\qquad \L^0=\L(E)\qquad\text{and}\\
W^{-1}=\xi W(E)\;,\qquad
W^{0}=W(E)\oplus (\L(E)\xi\partial_{\xi})\;,\qquad
W^{1}=\L(E)\partial_{\xi}\;.
\end{multline*}
Moreover $\gs_{\max}^{p}=(0)$ if $|p|\geq 2$ and
\beq
\label{case(ii)max}
\begin{split}
\gs_{\max}^{-1}&=(\gs\otimes\xi\L(E))\oplus(\out(\gs)\otimes\xi\L(E))\oplus(\xi W(E)),
\\
\gs_{\max}^{0}&=(\gs\otimes\L(E)) \oplus (\out(\gs)\otimes\L(E))\oplus W(E)\oplus (\L(E)\xi\partial_{\xi}),\\
\gs_{\max}^{1}&=\L(E)\partial_{\xi}.\\
\end{split}
\eeq
To state the main result Theorem \ref{main1} of this section we define $$\out(\gs^\L):=\out(\gs)\otimes \Lambda(n)\niplus \textbf{1}\otimes W(n)$$ and note that any subalgebra $\gg$ of the form
\begin{align}
\label{formageneral}
\notag\gs^{\L}=\gs\otimes \Lambda(n) \subset \gg\subset \gs_{\max}&=\oder(\gs)\otimes\Lambda(n)\niplus\textbf{1}\otimes W(n)\\
&=\gs^{\L}\niplus \out(\gs^\L)\;,
\end{align}
can be written as $\gg=\gs^\L\niplus F$ for a subalgebra $F$ of $\out(\gs^{\L})$. The algebra $\gg$ is admissible if and only if $F$ is admissible, that is the subalgebra $\pi(F)$ of $W$ is admissible.
\begin{theorem}
\label{main1}\label{main2}
Let  $\mathfrak{g} = \bigoplus_{p=-1}^{\ell} \mathfrak{g}^p$ be a  transitive  nonlinear  irreducible  $\mathbb{Z}$-graded  Lie  superalgebra of  depth $1$.  
Then $\gg$ is a semisimple Lie superalgebra with the  socle
  $\mathrm{soc}(\mathfrak{g})$  given by   $\mathrm{soc}(\mathfrak{g})= \mathfrak{s}^{\Lambda}=\gs\otimes\L$ where  $\mathfrak{s}$ is a uniquely  determined  simple Lie  superalgebra  and $\Lambda = \Lambda(n)$ is  the Grassmann   algebra for some nonnegative integer $n$.   
The superalgebra $\mathfrak{g}$  is   a graded  subalgebra   of the  Lie  superalgebra  
	$$\oder(\mathfrak{s}^{\Lambda}) = \mathfrak{s}^\Lambda  \niplus \mathrm{out}(\mathfrak{s}^{\L})\;,$$
	where $\mathrm{out}(\mathfrak{s}^{\L})=  \mathrm{out}(\mathfrak{s})\otimes \L \niplus 1 \otimes W$,
    with    one of  the  following  gradings:
\begin{itemize}
\item[(I)] $\gs=\bigoplus\gs^p$ has a $\bZ$-grading of depth $1$ and $\L=\L^0$ and $W=W^0$ have the trivial gradings of type $\vec k=(0)$;
\item[(II)] $\gs=\gs^0$ has the trivial grading  and $\L$ and $W$ have the  gradings of type $\vec k=(-1,0)$ with $\dim U^{-1}=1$.
\end{itemize}
Moreover $\gg$ can be  written as  a semidirect sum $\mathfrak{g} = \mathfrak{s}^{\Lambda} \niplus  F$,  where $F=\bigoplus F^p$ is a  nonnegatively graded   subalgebra  of  the Lie  superalgebra $\mathrm{out}(s^{\Lambda})$ which is  admissible, that is    the  natural projection from $F$ to $\partial_V$ is  surjective.

Conversely any grading as in (I) or (II) defines a transitive nonlinear and irreducible grading of depth $1$ of the Lie superalgebra $\gg=\gs^\L\niplus F$ for any nonnegatively graded subalgebra $F$ of $\out(\gs^\L)$ which is admissible.
\end{theorem}
\begin{proof}
By Theorem \ref{block} and Proposition \ref{graduato} any transitive nonlinear irreducible graded Lie superalgebra $\gg=\bigoplus\gg^p$ of depth $1$ is an admissible intermediate Lie superalgebra \eqref{formageneral} and therefore $\gg=\gs^{\L}\niplus F$ for an admissible graded subalgebra
$F$ of $\out(\gs^\L)$.
By Proposition \ref{firsthm} and its proof the grading of $\gg$ is induced by a grading of $\gs_{\max}$
either generated by the grading (I) with $\gg^{-1}=\gs^{-1}\otimes\L$ or by the grading (II) with $\gg^{-1}=\gs\otimes\xi\L(E)$. This fact and equations \eqref{case(i)max}-\eqref{case(ii)max} immediately imply that $F$ is graded in nonnegative degrees. 

Conversely let $\gs_{\max}$ be graded as in (I) or (II) and
$\gg=\gs^\L\niplus F$ the graded Lie superalgebra determined by an admissible subalgebra
$F$ of $\out(\gs^\L)$ which is graded in nonnegative degrees. It is clear that $d(\gg)=1$;
we now show that the grading of $\gg$ is also transitive, irreducible and nonlinear.
\vskip0.2cm\par\noindent
({\it Nonlinearity}) In case (I) $\gs^1\neq(0)$ from Lemma \ref{lemmetto} and $\gg^1\supset \gs^1\otimes\L\supset\gs^1$. In case (II) there exists by admissibility a vector $x\in\gg$ which projects to the constant supervector field $\partial_{\xi}$. Since $\gg$ is $\bZ$-graded one may assume $\deg(x)=\deg(\partial_{\xi})=1$.
\vskip0.2cm\par\noindent
({\it Transitivity}) In case (I) transitivity is a consequence of the inclusions
$$
\gg^{-1}=\gs^{-1}\otimes\L\;,\qquad\gg^0\subset (\oder^0(\gs)\otimes\L)\oplus W\;,\qquad \gg^p\subset \oder^p(\gs)\otimes\L\;\;(p\geq 1)
$$
and part (i) of Lemma \ref{lemmetto}. In case (II) the adjoint action of 
\begin{align*}
\gs^0_{\max}&=(\oder(\gs)\otimes\L(E))\oplus (\L(E)\xi\partial_{\xi})\oplus W(E)\\
&\simeq((\oder(\gs)\oplus\xi\partial_{\xi})\otimes\L(E))\oplus W(E)
\end{align*} 
on $\gg^{-1}\simeq(\gs\xi)\otimes\L(E)$ is given by the action on $\gs\xi$ of the ``conformal extension'' $\oder(\gs)\oplus\xi\partial_{\xi}$ of $\oder(\gs)$ and the natural left actions on $\L(E)$ of $\L(E)$ and $W(E)$; this easily implies transitivity for $\gg^0\subset\gs^0_{\max}$. Transitivity for $\gg^1$ is immediate.
\vskip0.2cm\par\noindent
({\it Irreducibility}) We start with (I) and a nonzero submodule $\gm\subset\gs^{-1}\otimes\L$ for $\gg^0$. 
By (v) of Lemma \ref{lemmetto}, it is always possibile to find $x\in\gm$ and $y\in\gs^0\otimes\L\subset\gg^0$ such that $[x,y]\in\gm$ is nonzero and of top degree, i.e. $x\in\gs^{-1}\otimes\L^{(n)}$
where here $\L=\bigoplus\L^{(p)}$ denotes the usual principal grading of the Grassmann algebra $\L$.

On the other hand $\gg^{-1}$ is isomorphic to the sum of $2^n$ 
copies of $\gs^{-1}$ as an $\gs^0$-module and if $\gm$ contains a nonzero element of a copy then it includes the full copy too, by (iii) of Lemma \ref{lemmetto}. Hence $\gm\supset\gs^{-1}\otimes\L^{(n)}$. 

By admissibility and $\deg(W)=0$ any constant supervector field is related to an element of $\gg^0$; we denote by $\mathcal F\subset \gg^0$ the collection of all these elements. 
 Since $\gm\supset\gs^{-1}\otimes\L^{(n)}$ and every $x\in\mathcal F$ is of the form \eqref{related}, one gets that $\gm$ contains a nonzero element also in every copy $\gs^{-1}$ of $\gs^{-1}\otimes\L^{(n-1)}$ and $$\gm\supset \bigoplus_{p\geq n-1}\gs^{-1}\otimes\L^{(p)}\;.$$ 
A repeated application of this argument yields 
$\gm\supset\gs^{-1}\otimes\L=\gg^{-1}$ 
and proves irreducibility.
\vskip0.1cm\par
The proof is similar in case (II) where $\cF$ is now replaced by the collection  of all the elements of $\gg^0$ which are related to a constant supervector field in $W(E)$. We omit the details.
\end{proof}
\begin{remark}
The proof of the irreducibility in Theorem \ref{main1} works also for those $\gg^0$-submodules of $\gg^{-1}$ which are  not necessarily $\bZ_2$-graded. Therefore the representation $\ad_{\gg^0}|_{\gg^{-1}}$
is always an  irreducible representation of $G$-type. 
\end{remark}
\vskip0.3cm\par
Theorem \ref{main1} reduces the description of the transitive nonlinear irreducible $\bZ$-graded Lie superalgebras $\gg$ of depth one to the description of the nonnegatively graded subalgebras $F=F^0\oplus F^1\oplus\cdots$ of $\out(\gs^\L)=\out(\gs)\otimes\L\niplus 1\otimes W$ with admissible subalgebra $\pi(F)\subset W$, where $\oder(\gs^\L)$ has the gradings  (I) or (II). Let us assume $\gs\neq\mathfrak{psl}(2|2)$ for simplicity of exposition.
\vskip0.2cm\par
In the grading (I) there are the two extreme cases $\out(\gs)=(0)$ and $U=(0)$. If $\out(\gs)=(0)$ then $\gg=\gs^\L\niplus F$ for {\it any} admissible subalgebra $F$ of $W$. Such   subalgebras can be  described  as  follows.  Let
  $$  \varphi :  \partial_{V} \longrightarrow  W=\partial_V\oplus \L^+\partial_V\;,  \qquad\varphi(\partial_{\xi^i} )= \partial_{\xi^i}  + \sum_{j} f_i^j \partial_{\xi^j}\;,$$
 be   a  section  of  the  natural projection  from $W=\partial_V\oplus \L^+\partial_V$ onto $\partial_V$, where $f_i^j \in (\Lambda^+)_{\0}$ is an even nilpotent superfunction on the purely odd $n$-dimensional linear supermanifold with coordinates $\left\{\xi^i\right\}$, for all $i,j=1,\ldots,n$.
 Then $F \subset W$ is  any  Lie  superalgebra  which   contains $\varphi(\partial_V)$, in particular,  the Lie subalgebra generated  by
 $\varphi(\partial_V)$ in $W$.
 
If $U=(0)$ then
$\gg\subset \oder(\gs)$ and $\gg=\gs\niplus F$ for {\it any} $\bZ$-graded subalgebra $F$ of $\out(\gs)$ (for this case see also Table \ref{tableouter}, \S \ref{basicgrad}-\S\ref{cartangrad} and Theorem \ref{importantcase}).

In general there is an exact sequence of $\bZ$-graded Lie superalgebras
\beq
\label{extension1}
0\longrightarrow F\cap \out(\gs)\otimes \L\longrightarrow F\longrightarrow F'\longrightarrow0\;,  
\eeq
where $F'$ is an admissible subalgebra of $W$. The first term of \eqref{extension1} can also be in turn described through the exact sequence,
\beq
\label{extension2}
0\longrightarrow F\cap \out(\gs)\otimes\L^+\longrightarrow F\cap \out(\gs)\otimes \L
\longrightarrow F''\;,
\eeq
where $F''$ is a $\bZ$-graded subalgebra of $\out(\gs)$. 
\vskip0.2cm\par
In the grading (II) of Theorem \ref{main1} there is the extreme case $E=(0)$ where
\begin{multline}
\label{gener}
\gg=\gg^{-1}\oplus\gg^0\oplus\gg^1\qquad\text{and}\\
\gg^{-1}=\gs\xi\;,\qquad \gg^0=\gs\niplus \wt F\;,\qquad \gg^1=\mathbb{C}\partial_{\xi}\;.\;\;\;
\end{multline}
Here $\wt F$ is {\it any} subalgebra of the direct sum $\out(\gs)\oplus\mathbb{C}\xi\partial_{\xi}$ of $\out(\gs)$ with the space generated by the grading operator $-\xi\partial_{\xi}$ of \eqref{gener}. We note that the Lie superalgebra \eqref{gener} is a generalization of the Lie superalgebra ``$G^z$'' and, at the same time, the Lie superalgebra ``$H^\xi$'' introduced in \cite[p. 71]{K1}. 

In general there is a graded ideal 
\beq
\label{ideal}
\mathfrak{i}=(\out(\gs)\otimes\L(E))\oplus(\L(E)\xi\partial_{\xi})\oplus(\L(E)\partial_{\xi})
\eeq 
of $\out_{0}(\gs^\L)=\bigoplus_{p\geq 0}\out^{p}(\gs^\L)$ and an exact sequence
\beq
\label{extension3}
0\longrightarrow F\cap \mathfrak{i}\longrightarrow F\longrightarrow F'\longrightarrow0\;,  
\eeq
where $F'$ is an admissible subalgebra of $W(E)$. These subalgebras can be described similarly as above in terms of sections $\varphi':\partial_E\longrightarrow W(E)$ of the natural projection of $W(E)=\partial_E\oplus\L^+(E)\partial_E$ onto $\partial_E$. If $\mathfrak{i}^+$ is the ideal of $\mathfrak{i}$ obtained by replacing $\L(E)$ with $\L^+(E)$ in \eqref{ideal} then the first term of \eqref{extension3} fits into the exact sequence
\beq
\label{extension4}
0\longrightarrow F\cap \mathfrak{i}^+\longrightarrow F\cap \mathfrak{i}
\longrightarrow F''\;,
\eeq
where $F''=(F'')^0\oplus\mathbb{C}\partial_{\xi}$ and $(F'')^0$ a is subalgebra of $\out(\gs)\oplus\mathbb{C}\xi\partial_{\xi}$.
\medskip\par
The extension problems associated with these exact sequences look rather complicated and will not be addressed here.

\end{document}